\documentclass[a4paper,11pt,twoside,leqno]{article}
\usepackage{amsmath,amssymb,amsfonts,amsthm,graphicx,fancyhdr}
\usepackage[utf8]{inputenc}
\usepackage[T1]{fontenc}
\usepackage{rotating}
\usepackage{xfrac}
\usepackage{thmtools}

\usepackage{hyperref}
\hypersetup{
	colorlinks=true,
	linkcolor=blue,
	filecolor=magenta,
	urlcolor=cyan,
}
\usepackage[showrefs]{refcheck}

\newtheorem{teo}{Theorem}[section]
\newtheorem{lem}[teo]{Lemma}
\newtheorem{prop}[teo]{Proposition}
\newtheorem{cor}[teo]{Corollary}
\newtheorem{dfn}[teo]{Definition}
\newtheorem{ques}[teo]{Question}
\newtheorem{teoi}{Theorem}

\declaretheoremstyle[
  spaceabove=\topsep, spacebelow=\topsep,
  headfont=\bf,  
  notefont=\mdseries, notebraces={(}{)},
  bodyfont=\rmfamily, 
  postheadspace=1em,
  qed=$\Diamond$
]{drem}
\declaretheorem[style=drem, name=Remark, numberlike=teo]{rmk}
\declaretheorem[style=drem, name=Example, numberlike=teo]{exa}

\newcommand{\eg}[0]{\emph{e.g.} }
\newcommand{\ie}[0]{\emph{i.e.} }

\newcommand{\srl}[1]{\overline{#1}}

\DeclareFontFamily{T1}{mafra}{}
\DeclareFontShape{T1}{mafra}{m}{n}{<->s*[0.95]yswab}{} 
\DeclareFontShape{T1}{mafra}{m}{it}{<->s*[1.0]ygoth}{} 
\DeclareTextFontCommand{\textgoth}{\yfrak}

\DeclareSymbolFont{mafrak}{T1}{mafra}{m}{n}
\DeclareSymbolFontAlphabet{\mathfr}{mafrak}

\DeclareSymbolFont{mbbold}{U}{bbold}{m}{n}
\DeclareSymbolFontAlphabet{\mathbbold}{mbbold}

\newcommand{\tx}[1]{\text{#1}}

\newcommand{\surj}[0]{\twoheadrightarrow}
\newcommand{\imp}[0]{\Rightarrow}

\newcommand{\rr}[0]{\ensuremath{\mathbb{R}}}
\newcommand{\zz}[0]{\ensuremath{\mathbb{Z}}}
\newcommand{\nn}[0]{\ensuremath{\mathbb{N}}}

\newcommand{\cc}[0]{\ensuremath{\mathbb{C}}}

\newcommand{\pp}[0]{\ensuremath{\mathbb{P}}}

\newcommand{\img}[0]{\mathrm{Im}\,}

\newcommand{\spt}[0]{\mathrm{supp}\,}

\newcommand{\un}[0]{\mathbbold{1}}

\newcommand{\tosp}[1]{\overset{#1}{\to}}

\newcommand{\comp}[0]{\mathsf{c}}

\newcommand{\eqtag}[0]{\addtocounter{teo}{1} \tag{\theteo}}

\begin{document}

\renewcommand{\thefootnote}{\fnsymbol{footnote}}

\renewcommand{\thefootnote}{\arabic{footnote}}

\newcommand{\ud}{\tfrac{1}{2}}
\newcommand{\ut}{\tfrac{1}{3}}
\newcommand{\uq}{\tfrac{1}{4}}

\newcommand{\lmd}{\lambda}
\newcommand{\Lmd}{\Lambda}
\newcommand{\Gm}{\Gamma}
\newcommand{\gm}{\gamma}
\newcommand{\GM}{\Gamma}

\newcommand{\bsl}\backslash
\newcommand{\acts}{\curvearrowright}
\newcommand{\donc}{\rightsquigarrow}

\newcommand{\smdd}[4]{\big( \begin{smallmatrix}#1 & #2 \\ #3 & #4\end{smallmatrix} \big)}

\newcommand{\BH}[0]{\mathsf{B}\mathcal{H}}
\newcommand{\M}[0]{\mathcal{M}}
\newcommand{\U}[0]{\mathcal{U}}
\newcommand{\G}[0]{\mathcal{G}}
\newcommand{\sgn}{\textrm{sgn}\,}

\begin{center}
\Large Harmonic projections, recurrent extensions and\\ stability of the Liouville property
\vspace*{1cm}

\centerline{\large Antoine Gournay \footnote{The author gratefully acknowledges partial support from the ERC AdG grant 101097307.}}
\end{center}

\vspace*{1cm}

\centerline{\textsc{Abstract}}

\begin{center}
\parbox{10cm}{{ \small 
The Liouville property is a strong form of amenability, but contrary to amenability, it is not well-behaved under extensions.
In this paper it is shown that, for every probability measure whose support generates the group as a semigroup, bounded harmonic functions are constant on the cosets of the FC-hypercentre, and more generally of any normal subgroup generated by elements whose centraliser is a recurrent set for the random walk; iterating this construction yields the hyperrecurrent subgroup.
Hence the Liouville property passes to hyperrecurrent extensions; this applies for instance to many metabelian groups, to Houghton's groups and to permutational wreath products.
For [FC-]hypercentral extensions, measures whose support only generates the group are also considered, using a projection from $\ell^\infty$ onto the space of harmonic functions.
Finally, it is shown that adding to a symmetric finitely supported measure a Dirac mass at an element of finite order whose centraliser is a recurrent set does not change the bounded harmonic functions.
\hspace*{.1ex} 
}}
\end{center}

\section{Introduction}\label{s-intro}

The Liouville property is said to hold for a group and some probability measure on that group if there are no non-constant bounded harmonic functions with respect to the measure.
One can trace back at least to Avez \cite{Avez72} the fact that the Liouville property implies amenability.
Among the four basic properties of amenable groups is that extensions of amenable groups are amenable.
However, this fails for the Liouville property, even for extensions of Abelian groups.
The classical example is the case of the lamplighter groups $(\zz/2\zz) \wr \zz^d$, which are extensions of an Abelian group by an Abelian group: Kaimanovich \& Vershik \cite{KV} showed that, for $d \geq 3$, they are not Liouville for any finitely supported symmetric measure whose support generates the group (whereas they are Liouville for such measures when $d \leq 2$).
Kaimanovich \& Vershik also conjectured that, for $d \geq 3$, the Poisson boundary is described by the limit configuration of the lamps; for a more complete history, see Lyons \& Peres' paper \cite{LP} settling the remaining cases of this Kaimanovich--Vershik conjecture.

In particular, the Liouville property depends on the measure: every amenable group, such as $\zz/2\zz \wr \zz^3$, admits a symmetric measure (with infinite support in general) whose support generates the group and for which it is Liouville (Kaimanovich \& Vershik \cite{KV}, Rosenblatt \cite{Ros}).
Whether it depends on the generating set, \ie whether a group which is Liouville for one finitely supported symmetric measure whose support generates it (for example a simple random walk) is Liouville for all such measures, is the so-called stability problem.
It is still open in general; see Erschler \& Frisch \cite{EF}, who give a positive answer for finitely generated linear groups of positive characteristic.

The first result is that [FC-]hypercentral extensions preserve the Liouville property, for suitable measures.
Recall that the FC-centre of a group is the subgroup of its elements which have a finite conjugacy class.
Iterating this construction (transfinitely) yields the upper FC-central series, whose final term is the FC-hypercentre; see Definition \ref{dfnFC} for the precise definitions.
The FC-hypercentre contains the hypercentre, which is obtained in the same way from the centre.
Recall also that a probability measure $\mu$ on $\Gm$ is {\bfseries non-degenerate} if its support generates $\Gm$ as a semigroup (\eg if $\mu$ is symmetric and its support generates $\Gm$), and that $\mu$ is a {\bfseries lazy version} of a measure $\nu$ if $\mu = t\delta_e + (1-t)\nu$ for some $t \in [0,1[$, where $\delta_e$ is the Dirac mass at the identity; $\mu$ and $\nu$ then have the same harmonic functions.
\begin{teoi}\label{teomain}
Let $\Gm$ be a countable group, let $N\lhd \Gm$ be contained in the FC-hypercentre of $\Gm$ and let $\phi : \Gm \to \Gm/N$ be the quotient map.
\begin{itemize}
\item[(a)] If $\mu$ is a non-degenerate probability measure on $\Gm$, then every bounded $\mu$-harmonic function on $\Gm$ is constant on the $N$-cosets.
In particular, $\Gm$ is $\mu$-Liouville if and only if $\Gm/N$ is $\phi_*\mu$-Liouville.
\item[(b)] If $\nu$ is a probability measure on $\Gm/N$ whose support generates $\Gm/N$ (as a group), then there are probability measures $\mu$ on $\Gm$ such that $\phi_*\mu$ is a lazy version of $\nu$ and $\Gm$ is $\mu$-Liouville if and only if $\Gm/N$ is $\nu$-Liouville.
\end{itemize}
\end{teoi}
Part (a) is Theorem \ref{teonondeg}.
Its proof relies on a simple argument, close to classical proofs of the Choquet--Deny theorem \cite{CD}: if $s$ belongs to the support of some convolution power of $\mu$, and if the random walk driven by $\mu$ visits the centraliser of $s$ infinitely often, then every bounded $\mu$-harmonic function is invariant under right multiplication by $s$ (Lemma \ref{lemmart}).
Part (b) does not require $\nu$ to be non-degenerate.
It is proved with harmonic projections, and the measures $\mu$ are fairly arbitrary, especially if $N$ is finitely generated and lies in the hypercentre; see Theorems \ref{teohyper} and \ref{teofchyper} and Remark \ref{remmuexist} for details.
The only requirement relating $\mu$ and $\nu$ is that the push-forward of $\mu$ by the quotient map is a convex combination of $\nu$ and the Dirac mass at the identity (\ie a lazier version of $\nu$).
An amusing tool introduced in the proof of part (b) is the harmonic projection (with respect to some measure $P$).
This is a (norm 1) projection from $\ell^\infty(\Gm)$ to the space of bounded $P$-harmonic functions.

When $N = \Gm$, \ie when $\Gm$ is FC-hypercentral, part (a) is a theorem of Jaworski \cite{Jaw}: FC-hypercentral groups are $\mu$-Liouville for every non-degenerate probability measure $\mu$ (and even for every probability measure whose support generates the group).
More generally, part (a) can be deduced from Jaworski's arguments; see the beginning of Section \ref{s-recext}.
Conversely, Frisch, Hartman, Tamuz \& Vahidi Ferdowsi \cite{FHTF} showed that a countable group which is not FC-hypercentral (equivalently, which has a non-trivial quotient with infinite conjugacy classes) carries a symmetric measure, whose support generates the group, for which it is not Liouville.

Elements with a finite conjugacy class have a centraliser of finite index, which the random walk visits infinitely often.
Part (a) therefore extends to the subgroups generated by elements whose centraliser is a recurrent set for the random walk, \ie is visited infinitely often, almost surely, by the random walk started at the identity.
\begin{teoi}\label{teomainrec}
Let $\mu$ be a non-degenerate probability measure on a countable group $\Gm$, let $N \lhd \Gm$ be generated by elements whose centralisers are recurrent sets for the random walk driven by $\mu$, and let $\phi : \Gm \to \Gm/N$ be the quotient map.
Then every bounded $\mu$-harmonic function on $\Gm$ is constant on the $N$-cosets.
In particular, $\Gm$ is $\mu$-Liouville if and only if $\Gm/N$ is $\phi_*\mu$-Liouville.
\end{teoi}
This is a special case of Theorem \ref{teorec}, which also allows one to iterate the construction, as for the FC-central series: bounded $\mu$-harmonic functions are constant on the cosets of a normal subgroup $\widehat{R}_\mu$, the $\mu$-hyperrecurrent subgroup, which contains the FC-hypercentre (Definition \ref{dfnRmu}).
In other words, the Liouville property passes to $\mu$-hyperrecurrent extensions, \ie to extensions by normal subgroups contained in $\widehat{R}_\mu$.
Here, as in Theorem \ref{teomain}(a), the measure $\mu$ is quite generic: the assumption of Theorem \ref{teomainrec} only involves the random walks induced by $\mu$ on the coset spaces $C_\Gm(s)\backslash\Gm$.
For symmetric finitely supported measures, it only depends on the corresponding Schreier graphs, and not on the choice of the measure (so that, for finitely generated groups, one can simply speak of hyperrecurrent extensions); and when it holds for such measures, it also holds for every symmetric measure with finite second moment whose support generates $\Gm$ (Proposition \ref{propindep}).
Theorem \ref{teomainrec} applies for instance, with an Abelian quotient $\Gm/N$, to the lamplighter groups $L \wr \zz^d$ with $L$ Abelian and $d \leq 2$, to the Baumslag--Solitar groups $BS(1,n)$ and other groups $A \rtimes \zz^k$ with $k \leq 2$ (such as Sol), to the free metabelian group of rank 2, to the lampshufflers $\mathrm{FSym}(\zz^d) \rtimes \zz^d$ with $d \leq 2$ and to Houghton's groups $H_n$; these groups are therefore Liouville for every symmetric measure with finite second moment whose support generates them (Section \ref{s-exa}).
For most of these groups, this was already known, by various methods (see Remark \ref{rmkknown}); Theorem \ref{teomainrec} gives a uniform and elementary argument, which requires no moment condition beyond the recurrence of the relevant induced random walks.

Iterating Theorem \ref{teomainrec}, one gets groups, such as $(\zz/2\zz \wr \zz) \wr \zz$, whose FC-hypercentre is trivial but for which bounded harmonic functions are lifted from the trivial group only after several steps (Example \ref{exaiter}).
It also reduces the case of a permutational wreath product $L \wr_X G$ over a recurrent Schreier graph $X$ to the case of $G$; together with a theorem of Amir, Angel, Matte Bon \& Vir\'ag \cite{AAMV}, this gives Liouville groups of exponential growth from the action of the Basilica group on the boundary of the binary tree (Example \ref{exafrac}).
What matters is the size of the conjugacy classes, rather than the length of a series of normal subgroups: Theorem \ref{teomainrec} applies to iterated wreath products over $\zz$, but not to all polycyclic groups (Remark \ref{rmklimits}).
In the other direction, since Thompson's group $F$ is not Liouville for such measures (Kaimanovich \cite{Kai-F}), the centraliser of a non-trivial element of $F$ is never a recurrent set for the corresponding random walks.
Section \ref{s-exa} also compares recurrence, the condition $\widehat{R}_\mu = \Gm$ and the Liouville property with their ``inner'' analogues, defined with the random walks induced on the conjugacy classes (Remark \ref{rmkinner} and Question \ref{quesinner}).

Motivated by the freedom in the choice of the measures and by the stability problem, the paper closes with some remarks on generating measures (Section \ref{s-gen}).
If $P$ is symmetric and finitely supported with $e \in \spt P$, then $P^n$ has the same bounded harmonic functions as $P$ (Lemma \ref{lempow}), and its support contains the support of any other finitely supported measure $Q$ for $n$ large enough.
So the stability problem reduces to the question whether adding a Dirac mass at an element $c$, \ie passing from $P$ to $\mu_t := tP + (1-t)\delta_c$, preserves the Liouville property.
Lemma \ref{lemmart} shows that it does when the centraliser of $c$ is a recurrent set for the random walk driven by $\mu_t$, \eg when the conjugacy class of $c$ is finite (Proposition \ref{proprec}).
When $P$ is symmetric and finitely supported, recurrence for the random walk driven by $P$ is enough, and then the bounded harmonic functions do not change at all:
\begin{teoi}\label{teomainC}
Let $P$ be a symmetric finitely supported probability measure on a countable group $\Gm$ whose support generates $\Gm$, let $c \in \Gm$ be an element whose centraliser is a recurrent set for the random walk driven by $P$, and let $t \in ]0,1[$.
Then $\BH_{\mu'_t}(\Gm) = \BH_P(\Gm)$, where $\mu'_t := tP + \tfrac{1-t}{2}(\delta_c + \delta_{c^{-1}})$.
If moreover $c$ has finite order, then $\BH_{\mu_t}(\Gm) = \BH_P(\Gm)$, where $\mu_t := tP + (1-t)\delta_c$.
In particular, $\Gm$ is then $\mu'_t$-Liouville, resp. $\mu_t$-Liouville, if and only if it is $P$-Liouville.
\end{teoi}
This is Corollary \ref{correc}.
It relies on Proposition \ref{proprecP}, according to which the centraliser of $c$ is then also a recurrent set for the random walk driven by $\mu'_t$, and, if $c$ has finite order, for the random walk driven by $\delta_c * S_t$, where $S_t := (1-t)\sum_{i \geq 0} (tP)^i$; the latter measure has the same bounded harmonic functions as $\mu_t$ (Proposition \ref{propS}).
For $c$ of finite order $n$, the proof uses the fact that $(\delta_c * S_t)^n$ is a product of $n$ conjugates of $S_t$, which gives a positive operator (Propositions \ref{propS} and \ref{propequiv}), together with a Dirichlet principle for non-reversible Markov chains due to Gaudilli\`ere \& Landim \cite{GL}.

\textit{Acknowledgements:} During the preparation of this paper J.~Brieussel pointed out to the author that the main lemma of Section \ref{s-hyper} (Lemma \ref{lemprinc}) is a weak form of a very recent result of Hartman, Hru\v{s}kov\'a \& Segev \cite{HHO}.
Claude Opus 5.5 (Anthropic) was used to assist in drafting parts of this manuscript. The content was reviewed and substantially revised by the author, who is responsible for the final text.

\section{Harmonic projection}\label{s-proj}

Consider a countable group $\Gm$.
The convolution product is nothing else than multiplication in the group ring $\rr[\Gm]$.
If $\delta_x$ is the function with value 1 at $x$ and 0 elsewhere, then $\delta_x * \delta_y = \delta_{xy}$.
This extends linearly to any two finitely supported functions, and, by density, to any pair of functions $f \in \ell^p(\Gm)$ and $g \in \ell^q(\Gm)$ (if $p,q<\infty$) where $p$ and $q$ are Hölder conjugate.

When $p=\infty$, the norm-closure of finitely supported functions is a smaller space denoted $c_0$, the pre-dual of $\ell^1$. However, if one of the functions is finitely supported, then the other function can be arbitrary. Hence density may be used to define a convolution between $\ell^1$ and $\ell^\infty$ as well.
It is possible to consider larger spaces to define the convolution product.
For example, if the measure has finite $\tfrac{1}{q}$-moment, then it suffices for the other function to have a gradient in $\ell^p$ (see \cite[Remark 3.17]{Go-metab}).

For our purposes however, it will be very useful to note that, if $\lambda$ (resp. $\rho$) denotes the left-regular (resp. right-regular) representation, then, for any function $f:\Gm \to \rr$, $\delta_x*f = \lambda_x f$ and $f*\delta_x = \rho_{x^{-1}} f$.

Let $P$ be a probability measure on $\Gm$.
A function is said to be $P$-harmonic (or harmonic with respect to $P$) if $f*P = f$ (\emph{i.e.} $f*(\delta_e-P) = 0$ as a convolution; $\delta_e$ is the Dirac mass at the identity element $e$ of the group).
By a slight abuse of notation, $P^0$ will denote $\delta_e$
and $P^n$ will also be used to denote $P^{*n}$, that is the convolution of $n$ copies of $P$.
It is straightforward to see that a $P$-harmonic function is also $P^n$-harmonic.

\emph{The identity element in the support of $P$.}
It will often be required that the identity element $e$ belongs to the support of $P$, \ie that $P(e)>0$.
As far as harmonic functions are concerned, this is no loss of generality: for any $a \in ]0,1[$, a function is $P$-harmonic if and only if it is harmonic with respect to the lazy measure $a\delta_e + (1-a)P$.
However, this assumption
is needed whenever the convolution powers $P^n$ themselves, and not only their averages, are used; this is the case at the end of this section, in Theorems \ref{teohyper} and \ref{teofchyper}, and in Section \ref{s-gen}.
Without it, periodicity phenomena appear: for the simple random walk $P = \tfrac{1}{2}(\delta_1+\delta_{-1})$ on $\zz$, the support of $P^n$ only contains integers with the same parity as $n$, and the indicator function $\un_{2\zz}$ of $2\zz$ is $P^2$-harmonic but not $P$-harmonic.

In fact, a $P$-harmonic function is harmonic with respect to a larger set of measures, namely the convex combinations of all the $P^n$: $\mathcal{C} = \{ \sum_{i \geq 0} a_iP^i \mid a_i \geq 0$ and $\sum a_i =1 \}$. Note that $\mathcal{C}$ is bounded, norm-closed and convex.

On the other hand, one can also look at measures which are $P$-invariant. For a probability measure this is not possible (at least if the support of $P$ generates an infinite group), but it makes sense for a sequence:
a sequence of probability measures $\mu_n$ will be said to be {\bfseries approximately $P$-invariant} if $\mu_n*(\delta_e-P)$ tends to 0 in $\ell^1$-norm.

A particularly classical instance of a sequence of measures in $\mathcal{C}$ which is also approximately $P$-invariant is
$G_n := \displaystyle \frac{1}{n} \sum_{i=0}^{n-1} P^i$.
Indeed
\[
(\delta_e-P)*G_n = G_n*(\delta_e-P) = \tfrac{1}{n} (\delta_e-P^{n}).
\]
Let $f \in \ell^\infty(\Gm)$. The sequence $f_n := f*G_n$ almost never converges, but one may pick an ultrafilter $\U$ and get an ultralimit $\lim_\U f_n$.
By the above equality, this limit is a $P$-harmonic function (and if $f$ were to be $P$-harmonic, the limit, just like the whole sequence, is identical to $f$).

Instead of thinking in terms of ultrafilters, one can think of a convolution by a mean.
Denote the set of means by $\M(\Gm)$.
This is the set of elements $m \in (\ell^\infty)^*$ satisfying the following two properties:\\
(i) $m(\un) = 1$ where $\un \in \ell^\infty(\Gm)$ denotes the function taking value 1 everywhere;\\
(ii) $m(f) \geq 0$ if $f \geq 0$. \\
It is standard to check that the set of means is convex, bounded and weak$^*$-closed (hence weak$^*$-compact).
Any probability measure belongs to $\M(\Gm)$; in fact the (convex) set of finitely supported probability measures is weak$^*$-dense in $\M(\Gm)$. See for example Greenleaf's book \cite[p.6]{Greenleaf} for more details.

Convolution by a mean can be defined by the left- (or right-)regular representation $\lambda$ (resp. $\rho$)
as well as the inversion $f \mapsto \check{f}$
(defined by $\check{f}(x) = f(x^{-1})$).
First, recall that the pairing between $f \in \ell^p(\Gm)$ and $g \in \ell^q(\Gm)$ is given by $\langle f \mid g \rangle = \sum_{\gm \in \Gm} f(\gm) g(\gm)$.
Second, note that
$(f*g)(\gm) = \sum_x f(x)g(x^{-1}\gm) = \sum_y f(\gm y^{-1})g(y)$.
Hence $(f*g)(\gm)
= \langle f \mid \lambda_\gm( \check{g}) \rangle
= \langle \rho_{\gm^{-1}}(\check{f}) \mid g \rangle$.

Going back to our example of the sequences $f_n := f*G_n$, one may pick a weak$^*$ accumulation point $m \in \M(\Gm)$ of $(G_n)_{n \geq 1}$.
Then $(f*m)(\gm) := \langle \rho_{\gm^{-1}}(\check{f}) \mid m \rangle$ is nothing else than the (ultra)limit mentioned above.

Note also that it is possible to use duality to define a random walk with a mean instead of a measure.
The convolution of two means $m$ and $n$ is defined through their action on $\ell^\infty(\Gm)$: $m*n$ is the mean such that $f*(m*n) = (f*m)*n$ for every $f \in \ell^\infty(\Gm)$ (explicitly, $\langle g \mid m*n \rangle := \big((\check{g}*m)*n\big)(e)$).
This product is associative, since it corresponds to the composition of the operators $f \mapsto f*m$, so that iterating the convolution $*m$ gives a ``random walk''.
However, it is weak$^*$-continuous in its second variable only: if $\mu_\alpha \to m$ and $\nu_\beta \to n$ are nets of probability measures, the iterated limits of $\mu_\alpha * \nu_\beta$ depend in general on the order in which they are taken.
In other words, $m*n$ is one of the two Arens products on $\ell^1(\Gm)^{**} = \ell^\infty(\Gm)^*$, and these two products differ when $\Gm$ is infinite, since $\ell^1(\Gm)$ is then not Arens regular (Young \cite{Young}).

The shorthand $\BH_P(\Gm)$ will be used to denote the space of bounded $P$-harmonic functions.
\begin{lem}\label{meanconv}
Let $P$ be a probability measure and $\mathcal{C}$ as above the convex set of its convolutions.
Let $\mu_n$ be a sequence of probability measures and $m$ be a weak$^*$ accumulation point of this sequence.
Let $\pi: \ell^\infty(\Gm) \to \ell^\infty(\Gm)$ be defined by $\pi(f) =f*m$.
Then  $\pi$ is a linear map of norm 1 and\\
- if $\mu_n \in \mathcal{C}$, then $\pi$ restricted to $\BH_P(\Gm)$ is the identity.\\
- if $\mu_n$ is an approximately $P$-invariant sequence, then the image of $\pi$ is contained in $\BH_P(\Gm)$.
\end{lem}
\begin{proof}
First, note that $\pi$ is linear as the inversion $\check{f}$ as well as the right-regular representation are linear.
Second, for any $\gm \in \Gm$, $|\pi(f)(\gm)| \leq \|f\|_\infty$, since $m$ has norm 1 as an element of $\ell^\infty(\Gm)^*$.
Hence $\|\pi(f)\|_\infty \leq \|f\|_\infty$, with equality if $f$ is constant.

If $\mu_n \in \mathcal{C}$, then for any $P$-harmonic function $f$, $f*\mu_n = f$.
By definition of weak$^*$ convergence and the convolution, this holds for $f*m$ as well.
So $\pi(f) = f$ for any $f \in \BH_P(\Gm)$.

If $\mu_n$ is an approximately $P$-invariant sequence, then for any $f$, Hölder's inequality yields $\| f*\mu_n *(\delta_e-P)\|_\infty \leq \|f\|_\infty \| \mu_n*(\delta_e-P)\|_1 \to 0$.
The conclusion follows for $m$ again since the convolution is defined by the pairing.
\end{proof}

For the upcoming proposition, the shorthand $I$ will be used for the identity operator and $P$ for the convolution by $P$ on the right, $f \mapsto f*P$.
\begin{prop}\label{tproj}
Let $P$ be a probability measure, $\mu_n \in \mathcal{C}$ a sequence of approximately $P$-invariant measures and $m$ an accumulation point of $\mu_n$.
Then $\pi:f \mapsto f*m$ is a norm 1 projection from $\ell^\infty(\Gm)$ to $\BH_P(\Gm)$. Furthermore \\
- $\ker (I- \pi) = \img \pi = \ker (I-P)$,\\
- $\img (I-P) \subset \ker \pi = \img (I- \pi)$, and this inclusion is strict if $\langle \spt P \rangle$ is infinite and amenable,\\
- if $\langle \spt P \rangle$ is infinite, then $c_0(\Gm) \subset \ker \pi$ and hence $\ker P \subsetneq \ker \pi$.
\end{prop}
\begin{proof}
That $\pi$ is a projection of norm 1 onto $\BH_P(\Gm)$ follows directly from Lemma \ref{meanconv}.
$\ker (I- \pi) = \img \pi$ and $ \ker \pi = \img (I- \pi)$ are generic statements of projections.
Also $\ker(I-P)$ is by definition $\BH_P(\Gm)$ so $\img \pi = \ker(I-P)$.

Since $\mu_n \in \mathcal{C}$, $P*\mu_n = \mu_n*P$.
Hence $\pi(f*P) = \pi(f)*P = \pi(f)$.
By linearity $f*(\delta_e-P) = f - f*P$ is in the kernel of $\pi$.

If $\spt P$ generates an infinite group, then, for any finite set $A$, the mass of $P^n$ in $A$ tends to 0.
From this it follows that no non-zero finitely supported function belongs to $\ker (I-P)$.
Since $\mu_n *(\delta_e-P) \to 0$, $m$ is singular (\ie for any $f
\in c_0(\Gm)$, $\langle f \mid m\rangle =0$).
Hence $\ker \pi$ contains $c_0(\Gm)$.
Finitely supported functions can also be used to show that the inclusion $\ker P \subset \ker \pi$ is strict.

Assume now that $H := \langle \spt P \rangle$ is infinite and amenable.
The operator $I-P$ on $\ell^\infty(\Gm)$ is the adjoint of the operator $S : \varphi \mapsto \varphi - \varphi * \check{P}$ on $\ell^1(\Gm)$, where $\check{P}(x) = P(x^{-1})$.
The operator $S$ is injective: if $\varphi = \varphi*\check{P}$, then $\varphi = \varphi*\check{P}^n$ for all $n$, and $\varphi * \check{P}^n$ tends pointwise to 0 (by dominated convergence, since the mass of $\check{P}^n$ in any finite set tends to 0).
Reiter's condition for the amenable group $H$ provides finitely supported probability measures $\varphi_k$ on $H$ such that $\| \varphi_k - \varphi_k * \delta_h \|_{\ell^1} \to 0$ for every $h \in H$, whence $\| S \varphi_k\|_{\ell^1} \to 0$.
So $S$ is injective but not bounded from below, hence its image is not closed.
By the closed range theorem, the image of its adjoint $I-P$ is not closed either, whereas $\ker \pi$ is closed.
Hence the inclusion $\img (I-P) \subset \ker \pi$ is strict.
\end{proof}
\begin{rmk}\label{rmkclosure}
The norm closure of $\img(I-P)$ consists of the functions $f$ such that $\|f*G_n\|_\infty \to 0$.
Indeed, $(g - g*P)*G_n = \tfrac{1}{n}(g - g*P^n)$ for every $g \in \ell^\infty(\Gm)$, and conversely $f - f*G_n = (I-P)(f*H_n)$, where $H_n := \tfrac{1}{n}\sum_{i=1}^{n-1}\sum_{j=0}^{i-1} P^j$.
This closure is in general strictly smaller than $\ker\pi$, and even than the intersection of the $\ker \pi$ over all accumulation points $m$ of $(G_n)_n$, which consists of the functions $f$ such that $f*G_n \to 0$ pointwise.
For instance, for the simple random walk on $\zz$, the indicator function $f$ of $\bigcup_{k \geq 1} [4^k, 4^k + 2^k]$ satisfies $f*G_n \to 0$ pointwise, as the walk spends a vanishing proportion of its time in this sparse set; but $\|f*G_n\|_\infty$ does not tend to $0$, as the walk started in the middle of a long interval remains in it for a long time.
On the other hand, the weak$^*$ closure of $\img(I-P)$ is the whole of $\ell^\infty(\Gm)$ when $\langle \spt P \rangle$ is infinite, since the operator $S$ of the above proof is injective.
\end{rmk}

The group $\Gm$ will be said to be {\bfseries $P$-Liouville} if the space $\BH_P(\Gm)$ of bounded $P$-harmonic functions is reduced to the constant functions.
This is often called the weak Liouville property.
It is good to note before moving on that, if the support of $P$ does not generate $\Gm$, then $\Gm$ cannot be $P$-Liouville.
Indeed, let $H<\Gm$ be the (strict) subgroup of $\Gm$ generated by the support of $P$. Then, if $f$ is constant on the cosets $\gm H$ ($\gm \in \Gm$), $f$ is $P$-harmonic. However, $f$ need not take the same constant value on each coset (take for example for $f$ the indicator function of $H$).

Assume now that $e \in \spt P$. If $\Gm$ is $P$-Liouville, then $\| P^n*\delta_g - P^n\|_{\ell^1} \to 0$ for every $g \in \Gm$, see Kaimanovich \& Vershik \cite[\S{}4]{KV}.
This is one of the places where the assumption $e \in \spt P$ cannot be dropped: for the simple random walk on $\zz$ (for which $\zz$ is Liouville), the support of $P^n$ is contained in $n + 2\zz$ and that of $P^n*\delta_1$ in $n+1+2\zz$, so that $\|P^n*\delta_1 - P^n\|_{\ell^1} = 2$ for every $n$.
Consequently, every weak$^*$ accumulation point of $(P^n)_{n \geq 0}$, or of $(G_n)_{n\geq 1}$, is a (right-)invariant mean.
It is also easy to see that if $m$ is a (right-)invariant mean, then $f*m$ is a constant function: in the Liouville case, the harmonic projection is simply a projection onto the constant functions.

\section{Central and hypercentral extensions}\label{s-hyper}

The FC-central series of a group is defined analogously to its upper central series, the centre being replaced by the FC-centre.
\begin{dfn}\label{dfnFC}
Let $\Gm$ be a group with identity element $e$.
\begin{itemize}
\item The {\bfseries FC-centre} $Z_{FC}(\Gm)$ of $\Gm$ is the subgroup of elements with a finite conjugacy class:
\[
Z_{FC}(\Gm) := \{ x \in \Gm \mid \{ gxg^{-1} \mid g \in \Gm \} \tx{ is finite} \}.
\]
\item The {\bfseries upper FC-central series} $\big( Z^\lambda_{FC}(\Gm) \big)_{\lambda}$, indexed by the ordinals, is defined by transfinite induction: $Z^0_{FC}(\Gm) := \{e\}$, and, for every ordinal $\lambda$,
\[
Z^{\lambda+1}_{FC}(\Gm) := \phi_\lambda^{-1} \Big( Z_{FC}\big( \Gm/Z^\lambda_{FC}(\Gm) \big) \Big),
\]
where $\phi_\lambda : \Gm \surj \Gm/Z^\lambda_{FC}(\Gm)$ is the quotient map; finally $Z^\beta_{FC}(\Gm) := \bigcup_{\lambda<\beta} Z^\lambda_{FC}(\Gm)$ if $\beta$ is a limit ordinal.
\item This non-decreasing family of normal subgroups is constant from some ordinal $\alpha$ on; $Z^\alpha_{FC}(\Gm)$ is the {\bfseries FC-hypercentre} of $\Gm$.
The group $\Gm$ is {\bfseries FC-hypercentral} if it is equal to its FC-hypercentre.
\end{itemize}
Replacing ``finite'' by ``of cardinality 1'' in this definition (\ie replacing the FC-centre by the centre $Z(\Gm)$) gives the upper central series $\big( Z^\lambda(\Gm) \big)_{\lambda}$, whose final term is the {\bfseries hypercentre} of $\Gm$.
\end{dfn}
The FC-centre is a characteristic subgroup of $\Gm$, since automorphisms map conjugacy classes onto conjugacy classes; in particular it is normal, so that the quotients above make sense.
Note that $Z^1_{FC}(\Gm) = Z_{FC}(\Gm)$.
As the centre is contained in the FC-centre, a transfinite induction shows that $Z^\lambda(\Gm) \subset Z^\lambda_{FC}(\Gm)$ for every ordinal $\lambda$; in particular the hypercentre is contained in the FC-hypercentre.
This inclusion may be strict: the infinite dihedral group $D_\infty = \zz \rtimes \zz/2\zz$ has trivial centre, hence trivial hypercentre, while its FC-centre is the subgroup $\zz$, of index 2, so that $Z^2_{FC}(D_\infty) = D_\infty$.

A finitely generated group $\Gm$ satisfies $Z^n_{FC}(\Gm) = \Gm$ for some finite $n$ if and only if it is virtually nilpotent.
Also, a finitely generated subgroup of the hypercentre (resp. of the FC-hypercentre) is nilpotent (resp. virtually nilpotent);
see Robinson's book \cite[Chapter 2, p.~50, and \S{}4.3]{Rob} for more details.

Given a map $\phi: X \to Y$, recall that the push-forward $\phi_*\mu$ is the measure $\phi_*\mu(y) = \mu(\phi^{-1}(y))$.
Recall from the introduction that a lazy version $\mu = t\delta_e + (1-t)\nu$ of $\nu$ (with $t \in [0,1[$) has the same harmonic functions as $\nu$, since the convolution by $\mu$ is a convex combination of the identity and of the convolution by $\nu$.

The proofs will first be done for the central series in order for them to be more readable. The slightly more general case of the FC-central series will be dealt with afterwards.

\begin{lem}\label{lemprinc}
Let $\eta$ and $\zeta$ be two probability measures on $\Gm$ such that $\eta * \zeta = \zeta * \eta$, and let $\mu = s\eta + (1-s)\zeta$ for some $s \in ]0,1[$. Then a bounded $\mu$-harmonic function is $\eta$-harmonic and $\zeta$-harmonic.
\end{lem}
\begin{proof}
Without loss of generality, it may be assumed that $s = \tfrac{1}{2}$.
Indeed, the lazy version $\tfrac{1}{2}(\delta_e + \mu)$ of $\mu$ is equal to $\tfrac{1}{2}\eta' + \tfrac{1}{2}\zeta'$, where $\eta' := (1-s)\delta_e + s\eta$ and $\zeta' := s\delta_e + (1-s)\zeta$ are lazy versions of $\eta$ and $\zeta$ which still commute, and lazy versions do not change harmonicity.

So assume that $\mu = \tfrac{1}{2}\eta + \tfrac{1}{2}\zeta$.
Since $\eta$ and $\zeta$ commute, the binomial formula gives
\[
\mu^{*j} = \sum_{k+\ell = j} a_{k,\ell} \; \zeta^{*k}*\eta^{*\ell}
\qquad \tx{where} \qquad
a_{k,\ell} := \frac{1}{2^{k+\ell}}\binom{k+\ell}{k}.
\]
For fixed $k$, the sequence $\ell \mapsto a_{k,\ell}$ is unimodal. Indeed
\[
\frac{a_{k,\ell+1}}{a_{k,\ell}} = \frac{1}{2} \binom{k+\ell+1}{k} \binom{k+\ell}{k}^{-1} = \frac{k+\ell+1}{2(\ell+1)},
\]
so that $a_{k,\ell+1}$ is larger than $a_{k,\ell}$ if $\ell < k-1$, equal to it if $\ell = k-1$ and smaller if $\ell > k-1$.
In particular the maximum of $\ell \mapsto a_{k,\ell}$ is attained at $\ell = k$ and, for all $k, \ell \geq 0$,
\[\eqtag \label{maxbinom}
a_{k,\ell} \leq a_{k,k} = \frac{1}{4^k}\binom{2k}{k} = \prod_{i=1}^{k} \frac{2i-1}{2i} \leq \frac{1}{\sqrt{2k+1}}.
\]
The last inequality is the classical bound on the central binomial coefficient; it follows from
$\displaystyle \Big( \prod_{i=1}^{k} \frac{2i-1}{2i} \Big)^2 \leq \prod_{i=1}^{k} \frac{2i-1}{2i} \cdot \frac{2i}{2i+1} = \frac{1}{2k+1}$.

Let $G_n := \frac{1}{n} \sum_{j=0}^{n-1} \mu^{*j}$; this is a probability measure which belongs to the convex set $\mathcal{C}$ associated to $\mu$.
The aim is to show that the sequence $(G_n)_{n \geq 1}$ is approximately $\eta$-invariant.
For $0 \leq k \leq n-1$, let $\bar a_{k,\ell} := a_{k,\ell}$ if $0 \leq \ell \leq n-1-k$ and $\bar a_{k,\ell} := 0$ otherwise (in particular for $\ell = -1$ and $\ell = n-k$). Then
\[
n \, G_n*(\delta_e-\eta)
= \sum_{k=0}^{n-1} \zeta^{*k} * \sum_{\ell=0}^{n-k} \big( \bar a_{k,\ell} - \bar a_{k,\ell-1} \big) \, \eta^{*\ell}.
\]
Since the measures $\zeta^{*k}*\eta^{*\ell}$ are probability measures, the triangle inequality yields
\[
n \, \| G_n*(\delta_e-\eta) \|_{\ell^1}
\leq \sum_{k=0}^{n-1} \sum_{\ell=0}^{n-k} \big| \bar a_{k,\ell} - \bar a_{k,\ell-1} \big|.
\]
For fixed $k$, the inner sum is the total variation of the unimodal sequence $0, a_{k,0}, a_{k,1}, \ldots, a_{k,n-1-k}, 0$, which is twice its maximal value.
By \eqref{maxbinom}, it is at most $2/\sqrt{2k+1}$.
Since $\frac{1}{\sqrt{2k+1}} \leq \sqrt{2k+1} - \sqrt{2k-1}$ for $k \geq 1$, this yields
\[\eqtag \label{estiGn}
\| G_n*(\delta_e-\eta) \|_{\ell^1}
\leq \frac{2}{n} \sum_{k=0}^{n-1} \frac{1}{\sqrt{2k+1}}
\leq \frac{2 \sqrt{2n-1}}{n}
\leq \frac{2\sqrt{2}}{\sqrt{n}}.
\]
Hence the sequence $(G_n)_{n \geq 1}$ is approximately $\eta$-invariant.

Let $m$ be a weak$^*$ accumulation point of $(G_n)_{n \geq 1}$ in $\M(\Gm)$ and let $\pi : f \mapsto f*m$.
By Lemma \ref{meanconv} applied with $P = \mu$, $\pi$ restricted to $\BH_\mu(\Gm)$ is the identity (as $G_n$ is a convex combination of powers of $\mu$).
By Lemma \ref{meanconv} applied with $P = \eta$, the image of $\pi$ lies in $\BH_\eta(\Gm)$.
So for any bounded $\mu$-harmonic function $f$, $f = f *m$ and the latter is an $\eta$-harmonic function, \ie $\BH_\mu(\Gm) \subset \BH_\eta(\Gm)$.
Since $\zeta$ and $\eta$ play symmetric roles, the same holds for $\zeta$.
\end{proof}

Note that the above proof essentially reduces the argument to a (non-symmetric) random walk on $\zz^2$.

\begin{rmk}\label{remfiner}
The rate $n^{-1/2}$ in \eqref{estiGn} is optimal.
Indeed, let $\Gm = \zz^2$, $\eta = \delta_{(1,0)}$ and $\zeta = \delta_{(0,1)}$.
The measures $\zeta^{*k}*\eta^{*\ell} = \delta_{(\ell,k)}$ have pairwise disjoint supports, so that the triangle inequality in the above proof is an equality.
For $1 \leq k \leq \tfrac{n-1}{2}$, the maximum of $\ell \mapsto \bar a_{k,\ell}$ is $a_{k,k}$, and $a_{k,k} \geq \frac{1}{2\sqrt{k}}$ (this follows from $a_{1,1} = \frac{1}{2}$ and $\big( \frac{2i-1}{2i} \big)^2 \geq \frac{2i-2}{2i-1} \cdot \frac{2i-1}{2i} = \frac{i-1}{i}$ for $i \geq 2$).
Using $\frac{1}{2\sqrt{k}} \geq \sqrt{k+1}-\sqrt{k}$ (and taking into account the term $k=0$, for which the maximum is $a_{0,0} = 1$), one gets $\| G_n*(\delta_e-\eta) \|_{\ell^1} \geq \sqrt{2/n}$ for every $n \geq 1$.
In fact, by Stirling's formula, $\sqrt{n} \, \| G_n*(\delta_e-\eta) \|_{\ell^1}$ tends to $2\sqrt{2/\pi} \approx 1.596$ in this example.
\end{rmk}

Here is a first amusing corollary of Lemma \ref{lemprinc}:
\begin{cor}\label{corfun}
If $\Gm$ is a countable Abelian group, then $\Gm$ is $\mu$-Liouville for any measure $\mu$ with $\Gm = \langle \spt \mu \rangle$.
\end{cor}
\begin{proof}
One could shorten the proof by quoting the case of finitely supported measures on finitely generated Abelian groups (see for example Woess' treatise \cite[Theorem 25.4]{Woe}).
However, this case also follows from Lemma \ref{lemprinc}, by induction on the number $n$ of elements of $\spt \mu \setminus \{e\}$: let us show that a bounded $\mu$-harmonic function $f$ is then invariant under multiplication by the subgroup generated by $\spt \mu$ (the case $n=0$ being trivial).
If $n=1$, then $\mu = (1-a)\delta_e + a\delta_x$ for some $x \in \Gm$ and $a \in ]0,1]$, and $f*\mu = f$ simply means that $f*\delta_x = f$, \ie that $f$ is invariant under multiplication by $x$.
If $n \geq 2$, pick $x \in \spt \mu \setminus \{e\}$ and write $\mu = (1-t)\mu' + t\delta_x$ with $t := \mu(x) \in ]0,1[$, where $\mu'$ is a probability measure with $n-1$ elements of its support outside $e$.
Since $\Gm$ is Abelian, $\mu'$ and $\delta_x$ commute, so that $f$ is $\mu'$-harmonic and $\delta_x$-harmonic by Lemma \ref{lemprinc}, and the induction hypothesis applies to $\mu'$.

Now let $\mu$ be arbitrary, let $f$ be a bounded $\mu$-harmonic function and let $B<\Gm$ be any finitely generated subgroup.
Since $\spt \mu$ generates $\Gm$, there is a finite set $F \subset \spt \mu$ such that $B$ is contained in the subgroup $\langle F \rangle$ generated by $F$.
If $\mu(F) = 1$, then $f$ is invariant under $\langle F \rangle$ by the finitely supported case.
Otherwise, $\mu$ is a non-trivial convex combination of $\beta := \mu(F)^{-1} \mu|_F$ and $\alpha := (1-\mu(F))^{-1} \mu|_{\Gm \setminus F}$.
Since $\Gm$ is Abelian, $\alpha * \beta = \beta * \alpha$, so that, by Lemma \ref{lemprinc}, $f$ is also $\beta$-harmonic, hence invariant under $\langle F \rangle$ by the finitely supported case.
So $f$ is constant on the $\langle F \rangle$-cosets, and in particular on the $B$-cosets.
Since $B$ can be any finitely generated subgroup of $\Gm$, the conclusion follows.
\end{proof}
Of course, Corollary \ref{corfun} is a special case of the Choquet--Deny theorem \cite{CD}; the point here is that it follows from Lemma \ref{lemprinc} alone.

\begin{lem}\label{lemcentre}
Let $Z < Z(\Gm)$ (hence $Z \lhd \Gm$) and let $\mu$ be a probability measure on $\Gm$ such that $\spt \mu \cap Z$ generates $Z$.
Then any bounded $\mu$-harmonic function $f$ is constant on the $Z$-cosets.
\end{lem}
\begin{proof}
If $Z$ is trivial there is nothing to prove. Otherwise $\spt \mu \cap Z$ is non-empty, so that $s := \mu(Z) > 0$; let $\zeta := s^{-1} \mu|_Z$.
If $s<1$, then $\mu = s\zeta + (1-s)\eta$ where $\eta := (1-s)^{-1} \mu|_{\Gm \setminus Z}$.
Note that $\zeta*\eta = \eta*\zeta$ since $\zeta$ is supported in $Z \subset Z(\Gm)$.
By Lemma \ref{lemprinc} (or trivially if $s=1$), $f$ is $\zeta$-harmonic.
For $\gm \in \Gm$, the function $z \mapsto f(\gm z)$ on $Z$ is then bounded and $\zeta$-harmonic, because $\zeta$ is supported on $Z$.
Since $Z$ is Abelian and $\spt \zeta = \spt \mu \cap Z$ generates $Z$, this function is constant by the Choquet--Deny theorem \cite{CD} (see also Corollary \ref{corfun}).
This means that $f$ is constant on the coset $\gm Z$.
\end{proof}

Note that there are finitely presented groups whose centre is an infinitely generated Abelian group (for example $\mathbb{Q}$; see Ould Houcine's paper \cite{OH}).

\begin{teo}\label{teohyper}
Let $\Gm$ be a countable group. Assume that $N\lhd \Gm$ is contained in the hypercentre of $\Gm$, let $N_\lambda := N \cap Z^\lambda(\Gm)$ for every ordinal $\lambda$, and let $\phi : \Gm \to \Gm/N$ be the quotient map.
Let $\nu$ be a probability measure on $\Gm/N$ whose support generates $\Gm/N$, and let $\mu$ be a probability measure on $\Gm$ such that $\phi_*\mu$ is a lazy version of $\nu$.
Assume that one of the following holds:
\begin{itemize}
\item[(a)] $N$ is finitely generated, the identity element $e$ belongs to $\spt \mu$, and $\spt \mu \cap N$ is a symmetric generating set of $N$;
\item[(b)] there is an $n \geq 1$ such that, for every ordinal $\lambda$, $\spt(\mu^n) \cap N_\lambda$ generates $N_\lambda$.
\end{itemize}
Then $\Gm$ is $\mu$-Liouville if and only if $\Gm/N$ is $\nu$-Liouville.
\end{teo}
\begin{proof}
One direction is obvious. If $\bar f:\Gm/N \to \rr$ is a non-constant bounded $\nu$-harmonic function, then $\bar f$ is also $\phi_*\mu$-harmonic (as $\phi_*\mu$ is a lazy version of $\nu$), and $\bar f \circ \phi$ is a non-constant bounded $\mu$-harmonic function on $\Gm$.

For the other direction, let us first check that (a) implies (b).
If $N$ is finitely generated, then $N$ is nilpotent, since it is a finitely generated subgroup of the hypercentre.
Finitely generated nilpotent groups satisfy the maximal condition on subgroups: all their subgroups are finitely generated and any non-decreasing sequence of subgroups is eventually constant.
Hence the non-decreasing family $(N_\lambda)_\lambda$ takes only finitely many distinct values, all of which are finitely generated.
Since $\spt \mu \cap N$ is a symmetric generating set of $N$, each element of a finite generating set of one of these subgroups is a product of elements of $\spt \mu$, hence belongs to $\spt(\mu^k) = (\spt \mu)^k$ for some $k$.
Since $e \in \spt \mu$, $\spt(\mu^k) \subset \spt(\mu^n)$ whenever $k \leq n$; so (b) holds for $n$ large enough.

So assume that (b) holds, and let $f$ be a bounded $\mu$-harmonic function; it is also $\mu^n$-harmonic.
Let us show by transfinite induction on $\lambda$ that $f$ is constant on the $N_\lambda$-cosets.
For $\lambda = 0$ there is nothing to prove.
If $f$ is constant on the $N_\lambda$-cosets, then $f = \bar f_\lambda \circ \phi_\lambda$, where $\phi_\lambda : \Gm \to \Gm/N_\lambda$ is the quotient map (note that $N_\lambda \lhd \Gm$) and $\bar f_\lambda$ is a bounded $(\phi_\lambda)_*\mu^n$-harmonic function on $\Gm/N_\lambda$.
The subgroup $N_{\lambda+1}/N_\lambda$ is central in $\Gm/N_\lambda$: for $x \in N_{\lambda+1}$ and $g \in \Gm$, the commutator $xgx^{-1}g^{-1}$ belongs both to $N$ and to $Z^\lambda(\Gm)$, hence to $N_\lambda$.
Furthermore $N_{\lambda+1}/N_\lambda$ is generated by $\phi_\lambda\big(\spt(\mu^n) \cap N_{\lambda+1}\big) \subset \spt\big((\phi_\lambda)_*\mu^n\big)$.
Thus Lemma \ref{lemcentre}, applied in $\Gm/N_\lambda$, shows that $\bar f_\lambda$ is constant on the $N_{\lambda+1}/N_\lambda$-cosets, \ie $f$ is constant on the $N_{\lambda+1}$-cosets.
If $\beta$ is a limit ordinal and $f$ is constant on the $N_\lambda$-cosets for all $\lambda < \beta$, then $f$ is constant on the $N_\beta$-cosets, since $N_\beta = \cup_{\lambda < \beta} N_\lambda$.

As $N$ is contained in the hypercentre, $N = N_\lambda$ for $\lambda$ large enough. Hence $f$ is constant on the $N$-cosets, \ie $f = \bar f \circ \phi$ for some bounded function $\bar f$ on $\Gm/N$.
Since $f$ is $\mu$-harmonic, $\bar f$ is $\phi_*\mu$-harmonic, hence $\nu$-harmonic.
So, if $\Gm/N$ is $\nu$-Liouville, $\bar f$ is constant, and so is $f$.
\end{proof}

\begin{rmk}\label{remmuexist}
It is not difficult to construct a measure $\mu$ with $\spt \mu \supset N$ that satisfies the hypotheses of Theorem \ref{teohyper}, and, where it makes sense, has finite moments.

First split $\mu = \mu_1+\mu_2$. Given $\nu$ on $\Gamma/N$, pick for each $\gm \in \spt \nu$ an element $\gm' \in \phi^{-1}(\gm)$ and let $\mu_2(\gm') = t\nu(\gm)$ (for some $t \in ]0,1[$).
If moment conditions are relevant, then $\Gm$ should be finitely generated. In that case it is always possible to pick the $\gm'$ so that their distance to the identity is (up to multiplicative constants which depend on the choice of generating set for $\Gm/N$) that of $\gm$.

Next let $\{\gm_i\}_{i \in \nn}$ be an enumeration of the elements of $N$.
Then set $\mu_1(\gm_i) = w(i)$ where $w$ is some positive function which sums to $1-t$.
If moment conditions are pertinent, let $\mu_0(\gm_i) = w(i) e^{-|\gm_i|}$ and then $\mu_1 = (1-t) \mu_0 / \|\mu_0\|_{\ell^1}$.

Then $\phi_*\mu$ is a lazy version of $\nu$ and $\spt \mu \supset N$, so that $\mu$ satisfies the hypotheses of Theorem \ref{teohyper} (with $n=1$ in (b)). If $\nu$ is symmetric, $\mu$ may also be chosen symmetric.
When $N$ is finitely generated, it is even simpler to construct such a measure: by (a), it suffices for $\mu_1$ to be a finitely supported measure whose support is a symmetric generating set of $N$ containing $e$.
\end{rmk}

Let us mention the statements in the FC-central case.
First, Corollary \ref{corfun} extends to FC-groups, provided the measure is invariant under conjugation:
\begin{cor}\label{corfunFC}
If $\Gm$ is a countable FC-group (\ie $Z_{FC}(\Gm) = \Gm$), then $\Gm$ is $\mu$-Liouville for any measure $\mu$ which is invariant under conjugation and such that $\Gm = \langle \spt \mu \rangle$.
\end{cor}
\begin{proof}
Let $f$ be a bounded $\mu$-harmonic function and let $B<\Gm$ be any finitely generated subgroup.
Since $\spt \mu$ generates $\Gm$ and is a union of conjugacy classes, which are finite, there is a finite union $F \subset \spt \mu$ of conjugacy classes such that $B$ is contained in $G := \langle F \rangle$.
If $\mu(F) = 1$, let $\beta := \mu$.
Otherwise, $\mu$ is a non-trivial convex combination of $\beta := \mu(F)^{-1} \mu|_F$ and $\alpha := (1-\mu(F))^{-1} \mu|_{\Gm \setminus F}$; as $\beta$ is invariant under conjugation, it commutes with every measure, in particular with $\alpha$, and $f$ is $\beta$-harmonic by Lemma \ref{lemprinc}.
Hence, for every $\gm \in \Gm$, $f_\gm : z \mapsto f(\gm z)$ is a bounded $\beta$-harmonic function on $G$, and it remains to show that $f_\gm$ is constant.

The group $G$ is a finitely generated FC-group, so its centre $Z(G)$ has finite index (it is the intersection of the centralisers of the finitely many generators).
Let $S := \bigcup_{k \geq 1} \spt(\beta^k)$ be the semigroup generated by $F$ and let $M$ be the exponent of the finite group $G/Z(G)$, so that $s^M \in S \cap Z(G)$ for every $s \in S$.
Write any $z \in Z(G)$ as a word in $F \cup F^{-1}$, replace each letter $s^{-1}$ by $s^{M-1}(s^M)^{-1}$ and move the central factors $(s^M)^{-1}$ to the right: this gives $z = w z'$ with $w \in S \cup \{e\}$ and $z' \in \langle S \cap Z(G) \rangle$.
Then $w = z z'^{-1} \in Z(G)$, so that $z \in \langle S \cap Z(G) \rangle$; \ie $Z(G) = \langle S \cap Z(G) \rangle$.
Now $f_\gm$ is $\beta^k$-harmonic for every $k$, so Lemma \ref{lemcentre}, applied in $G$ to $\beta^k$ and to the subgroup generated by $\spt(\beta^k) \cap Z(G)$, shows that $f_\gm$ is invariant under right multiplication by $\spt(\beta^k) \cap Z(G)$.
Hence $f_\gm$ is invariant under $Z(G)$ and induces a bounded harmonic function on the finite group $G/Z(G)$, for the image of $\beta$, whose support generates $G/Z(G)$.
By the maximum principle, this function is constant.
So $f$ is constant on the $G$-cosets, in particular on the $B$-cosets, and since $B$ is arbitrary, $f$ is constant.
\end{proof}
\begin{lem}\label{fccentral}
Let $Z < Z_{FC}(\Gm)$ be a subgroup which is normal in $\Gm$.
Let $\mu$ be a probability measure on $\Gm$ such that $\mu = \mu_1 + \mu_2$, where $\mu_1 \neq 0$ is supported on $Z$, is invariant under conjugation by $\Gm$, and its support generates $Z$.
Then any bounded $\mu$-harmonic function $f$ is constant on the $Z$-cosets.
\end{lem}
\begin{proof}
Let $\zeta := \|\mu_1\|_{\ell^1}^{-1} \mu_1$.
If $\mu_2 \neq 0$, then $\mu$ is a non-trivial convex combination of $\zeta$ and $\eta := \|\mu_2\|_{\ell^1}^{-1} \mu_2$, and $\zeta * \eta = \eta * \zeta$ because $\zeta$ is constant on conjugacy classes; so $f$ is $\zeta$-harmonic by Lemma \ref{lemprinc} (this is trivial if $\mu_2 = 0$).
For $\gm \in \Gm$, the function $z \mapsto f(\gm z)$ on $Z$ is then bounded and $\zeta$-harmonic.
Now $Z$ is a countable FC-group (its conjugacy classes are contained in those of $\Gm$), $\zeta$ is invariant under conjugation by $Z$ and its support generates $Z$.
By Corollary \ref{corfunFC}, this function is constant, \ie $f$ is constant on the coset $\gm Z$.
\end{proof}

Theorem \ref{teohyper} can be adapted to the FC-case:
\begin{teo}\label{teofchyper}
Let $\Gm$ be a countable group. Assume that $N\lhd \Gm$ is contained in the FC-hypercentre of $\Gm$, let $N_\lambda := N \cap Z^\lambda_{FC}(\Gm)$ for every ordinal $\lambda$, and let $\phi : \Gm \to \Gm/N$ be the quotient map.
Let $\nu$ be a probability measure on $\Gm/N$ whose support generates $\Gm/N$.
Let $\mu$ be a probability measure on $\Gm$ such that $\mu=\mu_1+\mu_2$, where $\mu_1$ is supported on $N$, $\phi_*\mu_2$ is a lazy version of $t\nu$ (with $t=\|\mu_2\|_{\ell^1} > 0$), and, for every ordinal $\lambda$, the push-forward to $\Gm/N_\lambda$ of the restriction of $\mu_1$ to $N_{\lambda+1}$ is invariant under conjugation and its support generates $N_{\lambda+1}/N_\lambda$.
Then $\Gm$ is $\mu$-Liouville if and only if $\Gm/N$ is $\nu$-Liouville.
\end{teo}
\begin{proof}
Since $\phi_*\mu_1$ is a multiple of the Dirac mass at the identity, $\phi_*\mu$ is a lazy version of $\nu$, and one direction is obvious, as in the proof of Theorem \ref{teohyper}.
For the other direction, let $f$ be a bounded $\mu$-harmonic function.
As in the proof of Theorem \ref{teohyper}, it suffices to show, by transfinite induction on $\lambda$, that $f$ is constant on the $N_\lambda$-cosets, the case $\lambda = 0$ and the case of limit ordinals being identical.
So assume that $f = \bar f_\lambda \circ \phi_\lambda$, where $\phi_\lambda : \Gm \to \Gm/N_\lambda$ is the quotient map and $\bar f_\lambda$ is a bounded $(\phi_\lambda)_*\mu$-harmonic function.
The subgroup $N_{\lambda+1}/N_\lambda$ is normal in $\Gm/N_\lambda$ and contained in its FC-centre: if $x \in N_{\lambda+1}$, the conjugates of $x$ lie in $N$ and in finitely many cosets of $Z^\lambda_{FC}(\Gm)$, and two of them lying in the same coset differ by an element of $N \cap Z^\lambda_{FC}(\Gm) = N_\lambda$.
Moreover $(\phi_\lambda)_*\mu = \mu'_1 + \mu'_2$, where $\mu'_1$ is the push-forward of the restriction of $\mu_1$ to $N_{\lambda+1}$; by assumption, $\mu'_1$ is supported on $N_{\lambda+1}/N_\lambda$, invariant under conjugation, and its support generates $N_{\lambda+1}/N_\lambda$.
If $N_{\lambda+1} = N_\lambda$ there is nothing to prove.
Otherwise $\mu'_1 \neq 0$, and Lemma \ref{fccentral}, applied in $\Gm/N_\lambda$, shows that $\bar f_\lambda$ is constant on the $N_{\lambda+1}/N_\lambda$-cosets, \ie $f$ is constant on the $N_{\lambda+1}$-cosets.
\end{proof}
Compared to Theorem \ref{teohyper}, no power of $\mu$ is needed here, but the part $\mu_1$ of $\mu$ has to be adapted to the filtration $(N_\lambda)_\lambda$.

\section{Extensions and recurrent centralisers}\label{s-recext}

This section relies on a different argument, which uses neither Lemma \ref{lemprinc} nor harmonic projections, and which simplifies several of the statements of Section \ref{s-hyper}.
It is close to classical proofs of the Choquet--Deny theorem \cite{CD}.

For elements with a finite conjugacy class, the results of this section are essentially known.
Indeed, Jaworski \cite[Lemma 4.7]{Jaw} showed (see also Azencott \cite[Th\'eor\`eme IV.1]{Aze}) that, for every non-degenerate probability measure $\mu$ on $\Gm$ (and, more generally, under a weaker assumption on $\mu$), bounded $\mu$-harmonic functions are invariant under right multiplication by the FC-centre.
Applied to successive quotients, this gives Theorem \ref{teonondeg} below, and Corollary \ref{corJaw} is a special case of Jaworski's theorem \cite[Theorems 4.5 and 4.8]{Jaw} recalled in the introduction (to which Frisch, Hartman, Tamuz \& Vahidi Ferdowsi \cite{FHTF} refer for the converse of their result).
The argument below does not use representation theory, and it also applies to elements whose centraliser is a recurrent set for the random walk (Lemma \ref{lemmart} and Theorem \ref{teorec}).
We are not aware of a general result of this kind, although special cases, such as lamplighter groups over recurrent groups, were known (see Remark \ref{rmkknown}).

Given a probability measure $\mu$ on $\Gm$, let $\mathcal{G}_\mu := \sum_{n \geq 0} \mu^n$ denote its {\bfseries Green measure}, which is a (possibly infinite) positive measure on $\Gm$.
A subgroup $H < \Gm$ will be called {\bfseries $\mu$-recurrent} if $\mathcal{G}_\mu(H) = \sum_{n \geq 0} \mu^n(H) = \infty$.

Let us recall the probabilistic meaning of this definition, which is only needed to check it in examples.
The random walk driven by $\mu$ is $X_0 := e$, $X_{n+1} := X_n Y_{n+1}^{-1}$, where the $Y_n$ are independent with law $\mu$; it satisfies $\mathbb{E}[g(xX_n)] = (g*\mu^n)(x)$ for every bounded function $g$ on $\Gm$ and every $x \in \Gm$.
Its image $(HX_n)_n$ in the set $H\backslash\Gm$ of right cosets of $H$ is a Markov chain, the random walk induced by $\mu$ on $H\backslash\Gm$, whose steps are $Hx \mapsto HxY^{-1}$, with $Y$ of law $\mu$; it is irreducible if $\mu$ is non-degenerate.
As $\mu^n(H)$ is the probability that $X_n \in H$, \ie that this induced random walk is at the coset $H$ at time $n$, the subgroup $H$ is $\mu$-recurrent if and only if the coset $H$ is a recurrent state of the induced random walk.
In particular, if the induced random walk is irreducible, then $H$ is $\mu$-recurrent if and only if this random walk is recurrent, and then $X_n \in H$ for infinitely many $n$, almost surely (see Woess \cite[Definition 1.14 and Proposition 1.17]{Woe}).
\begin{lem}\label{lemmart}
Let $P$ be a probability measure on a countable group $\Gm$, let $s \in \spt(P^k)$ for some $k \geq 1$, and assume that the centraliser $H$ of $s$ is $P$-recurrent; this holds for instance if the conjugacy class of $s$ is finite.
Then every bounded $P$-harmonic function $f$ satisfies $f(xs) = f(x)$ for all $x \in \Gm$.
\end{lem}
\begin{proof}
Let $f$ be a bounded $P$-harmonic function and let $h := f*\delta_s - f$, \ie $h(y) = f(ys^{-1}) - f(y)$.
Since $f$ is also $P^k$-harmonic, the function $q := f^2*P^k - f^2$ satisfies
\begin{align*}
q(y) &= \sum_{\gm} f(y\gm^{-1})^2 P^k(\gm) - 2 f(y) \sum_{\gm} f(y\gm^{-1}) P^k(\gm) + f(y)^2 \\
&= \sum_{\gm} \big( f(y\gm^{-1}) - f(y) \big)^2 P^k(\gm),
\end{align*}
so that $q \geq P^k(s)\, h^2$ (in particular, $f^2$ is subharmonic).
For every $N$, the sum $\sum_{n=0}^{N-1} q*P^n = \sum_{n=N}^{N+k-1} f^2*P^n - \sum_{n=0}^{k-1} f^2*P^n$ telescopes, and is thus at most $k\|f\|_\infty^2$.
Hence, as $G_n = \tfrac{1}{n}\sum_{i=0}^{n-1} P^i$ are normalised truncations of the Green measure $\mathcal{G}_P$,
\[\eqtag \label{greenbound}
h^2*\mathcal{G}_P \leq \frac{k\|f\|_\infty^2}{P^k(s)} \quad \textrm{pointwise}.
\]
In particular, $h^2*P^n \to 0$ pointwise and, by the Cauchy--Schwarz inequality, $|h*P^n| \leq (h^2*P^n)^{1/2} \to 0$, \ie $f*\delta_s*P^n \to f$ pointwise.

Now fix $x \in \Gm$, let $c := x s^{-1} x^{-1}$ and $w := \delta_{c^{-1}}*f - f$, \ie $w(y) = f(cy) - f(y)$.
As left and right convolutions commute, $w$ is a bounded $P$-harmonic function, and $w(x) = f(xs^{-1}) - f(x) = h(x)$.
Moreover, if $\eta \in H$, then $cx\eta = xs^{-1}\eta = x\eta s^{-1}$, so that $w(x\eta) = h(x\eta)$: the functions $w$ and $h$ coincide on $xH$.
In probabilistic terms, the optional stopping theorem, applied to the bounded martingale $\big(w(xX_n)\big)_n$ and to the first time after $m$ at which $xX_n \in xH$, expresses $w(x)$ as an average of values of $h$.
Let us write this argument with measures only.

Let $H^\comp := \Gm \setminus H$ and, for a measure $\kappa$ on $\Gm$ and a subset $A \subset \Gm$, let $\kappa|_A$ denote the restriction of $\kappa$ to $A$.
For a finite positive measure $\kappa$ on $\Gm$, let $\Phi(\kappa) := P*(\kappa|_{H^\comp})$.
For $m \geq 0$, let $\kappa_{m,0} := P^m$ and $\kappa_{m,j+1} := \Phi(\kappa_{m,j})$ for $j \geq 0$; in probabilistic terms, $\kappa_{m,j}|_H$ describes the first visit to $H$ after time $m$, when it occurs at time $m+j$.
These measures have the following properties.
\begin{itemize}
\item[(i)] Since $w$ is $P$-harmonic, $w*\kappa = w*(\kappa|_H) + (w*P)*(\kappa|_{H^\comp}) = w*(\kappa|_H) + w*\Phi(\kappa)$ for every $\kappa$; hence, for every $J \geq 0$,
\[
w = w*P^m = \sum_{j=0}^{J} w*(\kappa_{m,j}|_H) + w*\kappa_{m,J+1}.
\]
\item[(ii)] As the total mass of $\Phi(\kappa)$ is $\kappa(H^\comp)$, $\kappa_{m,j}(\Gm) = \kappa_{m,j}(H) + \kappa_{m,j+1}(\Gm)$; by telescoping,
\[
\sum_{j=0}^{J} \kappa_{m,j}(H) + \kappa_{m,J+1}(\Gm) = \kappa_{m,0}(\Gm) = 1.
\]
\item[(iii)] $\kappa_{m,j} \leq P^{m+j}$, by induction on $j$, as $\Phi(\kappa) \leq P*\kappa$.
\item[(iv)] $\kappa_{m,J}(\Gm) \to 0$ as $J \to \infty$; this is where the $P$-recurrence of $H$ is used (its proof is given after the main argument).
\end{itemize}
By (ii), $\theta_m := \sum_{j \geq 0} \kappa_{m,j}|_H$ is a positive measure supported on $H$, of total mass at most $1$.
Evaluating (i) at $x$ and letting $J \to \infty$, (iv) gives $w(x) = (w*\theta_m)(x)$, as $|(w*\kappa_{m,J+1})(x)| \leq \|w\|_\infty\, \kappa_{m,J+1}(\Gm)$.
As $\theta_m$ is supported on $H$ and $w = h$ on $xH$, $(w*\theta_m)(x) = \sum_{\gm \in H} w(x\gm^{-1})\, \theta_m(\gm) = (h*\theta_m)(x)$.
Hence, by the Cauchy--Schwarz inequality and (iii),
\[
w(x)^2 = (h*\theta_m)(x)^2 \leq \theta_m(\Gm)\, (h^2*\theta_m)(x) \leq \sum_{n \geq m} (h^2*P^n)(x),
\]
which tends to $0$ as $m \to \infty$, being the tail of the series $(h^2*\mathcal{G}_P)(x)$, which converges by \eqref{greenbound}.
So $w(x) = 0$, \ie $f(xs^{-1}) = f(x)$; as $x$ is arbitrary, $f$ is invariant under right multiplication by $s^{-1}$, hence by $s$.

It remains to prove (iv).
Recall from Section \ref{s-proj} that the pairing between a bounded function $g$ and a finite measure $\kappa$ on $\Gm$ is $\langle g \mid \kappa \rangle = \sum_{\gm} g(\gm)\, \kappa(\gm)$, and let $\check P(\gm) := P(\gm^{-1})$.
Then $\langle g \mid P*\kappa \rangle = \langle \check P*g \mid \kappa \rangle$, since both sides are equal to $\sum_{x,y \in \Gm} P(x)\, \kappa(y)\, g(xy)$.
Let $\Phi^* g := \un_{H^\comp} \cdot (\check P*g)$, so that $\langle g \mid \Phi(\kappa) \rangle = \big\langle \check P*g \,\big|\, \kappa|_{H^\comp} \big\rangle = \langle \Phi^* g \mid \kappa \rangle$; hence
\[
\kappa_{m,J}(\Gm) = \big\langle \un \,\big|\, \Phi^J(P^m) \big\rangle = \big\langle (\Phi^*)^J \un \,\big|\, P^m \big\rangle.
\]
As $\Phi^*$ is positive and $\Phi^* \un \leq \un$, the sequence $\big((\Phi^*)^J \un\big)_J$ is non-increasing; let $a$ be its pointwise limit (in probabilistic terms, $a(y)$ is the probability that the sequence $(Y_n \cdots Y_1 y)_{n \geq 0}$ never visits $H$).
By dominated convergence, $\Phi^* a = a$, \ie $a = 0$ on $H$ and $\check P*a = a$ on $H^\comp$, and $\kappa_{m,J}(\Gm) \to \langle a \mid P^m \rangle$ as $J \to \infty$.
Moreover, $\Phi^*$ commutes with the right translations by the elements of $H$, as left and right convolutions commute and $\un_{H^\comp}$ is invariant under these translations; so $a$ is invariant under right multiplication by $H$.
Hence $\psi := \check P*a - a$ is invariant under right multiplication by $H$ and vanishes on $H^\comp$, \ie $\psi = \psi(e)\, \un_H$, where $\psi(e) = (\check P*a)(e) \geq 0$.
It follows that $\langle \psi \mid P^n \rangle = \psi(e)\, P^n(H)$ for every $n$.
As $\langle \psi \mid P^n \rangle = \langle a \mid P^{n+1} \rangle - \langle a \mid P^n \rangle$, the sum of these terms telescopes, and, since $\langle a \mid P^0 \rangle = a(e) = 0$,
\[
\psi(e) \sum_{n=0}^{N-1} P^n(H) = \langle a \mid P^N \rangle \leq 1 \qquad \textrm{for every } N \geq 1
\]
(this is the telescoping argument of the proof of Woess \cite[Theorem 1.16]{Woe}, applied to the function $\un - a$, which satisfies $\check P*(\un - a) = \un - a - \psi \leq \un - a$).
Since $H$ is $P$-recurrent, this forces $\psi(e) = 0$, \ie $\check P*a = a$.
Hence $\kappa_{m,J}(\Gm) \to \langle a \mid P^m \rangle = \big\langle \check P^m * a \,\big|\, \delta_e \big\rangle = a(e) = 0$, which is (iv).

Finally, assume that the conjugacy class of $s$ is finite, so that $H$ has finite index $d$, and let $X := H\backslash\Gm$.
Functions on $X$ can be identified with the functions on $\Gm$ which are invariant under left multiplication by $H$, and right convolution by $P$ preserves the latter.
The corresponding operator $K$ on $\ell^2(X)$ is a convex combination of permutation operators, since right multiplications by elements of $\Gm$ are bijections of $X$; in particular, $K$ is a contraction.
By von Neumann's mean ergodic theorem, $\tfrac{1}{N}\sum_{n=0}^{N-1} K^n$ converges to the orthogonal projection $\Pi$ onto the space of $K$-invariant functions, which contains the constant function $\un_X$.
As $P^n(H) = (\un_H*P^n)(e) = \langle K^n\delta_H \mid \delta_H \rangle$, where $\delta_H \in \ell^2(X)$ is the indicator function of the coset $H$ and $\langle \cdot \mid \cdot \rangle$ is the scalar product of $\ell^2(X)$, one gets
\[
\frac{1}{N}\sum_{n=0}^{N-1} P^n(H) \tosp{N \to \infty} \langle \Pi\delta_H \mid \delta_H \rangle = \|\Pi\delta_H\|^2 \geq \big| \langle \delta_H \mid d^{-1/2}\un_X \rangle \big|^2 = \frac{1}{d},
\]
so that $\mathcal{G}_P(H) = \infty$, \ie $H$ is $P$-recurrent.
\end{proof}

Lemma \ref{lemmart} simplifies several of the results of Section \ref{s-hyper}.
\begin{teo}\label{teoFCmart}
Let $\mu$ be a probability measure on a countable group $\Gm$.
If $Z < Z_{FC}(\Gm)$ and $\spt \mu \cap Z$ generates $Z$, then every bounded $\mu$-harmonic function is constant on the cosets $xZ$.
In particular, a countable FC-group (\eg a countable Abelian group) is $\mu$-Liouville for every measure $\mu$ whose support generates it.
\end{teo}
\begin{proof}
Every $s \in \spt \mu \cap Z$ has a finite conjugacy class, so that, by Lemma \ref{lemmart}, every bounded $\mu$-harmonic function $f$ satisfies $f(xs) = f(x)$ for all $x \in \Gm$.
Hence $f$ is invariant under right multiplication by the subgroup generated by $\spt \mu \cap Z$, which is $Z$.
If $\Gm$ is an FC-group, this applies to $Z = \Gm$.
\end{proof}
Theorem \ref{teoFCmart} contains Lemmas \ref{lemcentre} and \ref{fccentral} and Corollaries \ref{corfun} and \ref{corfunFC}, without any assumption of conjugation invariance, normality or finite generation.

\begin{teo}\label{teonondeg}
Let $\Gm$ be a countable group, let $N \lhd \Gm$ be contained in the FC-hypercentre of $\Gm$ and let $\phi : \Gm \to \Gm/N$ be the quotient map.
Let $\mu$ be a non-degenerate probability measure on $\Gm$.
Then every bounded $\mu$-harmonic function is constant on the $N$-cosets, so that $\BH_\mu(\Gm) = \{ \bar f \circ \phi \mid \bar f \in \BH_{\phi_*\mu}(\Gm/N) \}$.
In particular, if $\phi_*\mu$ is a lazy version of a measure $\nu$ on $\Gm/N$, then $\Gm$ is $\mu$-Liouville if and only if $\Gm/N$ is $\nu$-Liouville.
\end{teo}
\begin{proof}
Let $f$ be a bounded $\mu$-harmonic function, and let $N_\lambda := N \cap Z^\lambda_{FC}(\Gm)$.
As in the proof of Theorem \ref{teofchyper}, it suffices to show that if $f = \bar f_\lambda \circ \phi_\lambda$ is constant on the $N_\lambda$-cosets, where $\phi_\lambda : \Gm \to \Gm/N_\lambda$ is the quotient map, then $f$ is constant on the $N_{\lambda+1}$-cosets.
Recall from that proof that $N_{\lambda+1}/N_\lambda$ is contained in the FC-centre of $\Gm/N_\lambda$, and that $\bar f_\lambda$ is $(\phi_\lambda)_*\mu$-harmonic.
Let $s \in N_{\lambda+1}/N_\lambda$.
Since $\mu$ is non-degenerate, $s$ belongs to the support of $\big((\phi_\lambda)_*\mu\big)^k = (\phi_\lambda)_*(\mu^k)$ for some $k$, and its conjugacy class in $\Gm/N_\lambda$ is finite; so Lemma \ref{lemmart} shows that $\bar f_\lambda$ is invariant under right multiplication by $s$.
Hence $\bar f_\lambda$ is constant on the $N_{\lambda+1}/N_\lambda$-cosets, \ie $f$ is constant on the $N_{\lambda+1}$-cosets.

Finally, a bounded function of the form $\bar f \circ \phi$ is $\mu$-harmonic if and only if $\bar f$ is $\phi_*\mu$-harmonic, and a lazy version of $\nu$ has the same harmonic functions as $\nu$.
\end{proof}
Thus, for non-degenerate measures, the hypotheses (a) and (b) of Theorem \ref{teohyper}, as well as the conditions on $\mu_1$ and $\mu_2$ in Theorem \ref{teofchyper}, are not needed.

\begin{cor}\label{corJaw}
A countable FC-hypercentral group is $\mu$-Liouville for every non-degenerate measure $\mu$ (Jaworski \cite[Theorems 4.5 and 4.8]{Jaw}).
\end{cor}
\begin{proof}
This is Theorem \ref{teonondeg} with $N = \Gm$.
\end{proof}

Lemma \ref{lemmart} also gives extension results beyond the FC-hypercentre.
For a probability measure $\mu$ on $\Gm$, we shall say that an element $s \in \Gm$ has a {\bfseries $\mu$-recurrent centraliser} if its centraliser $C_\Gm(s)$ is $\mu$-recurrent.
If $\mu$ is non-degenerate, this means that the random walk induced by $\mu$ on the set $C_\Gm(s)\backslash\Gm$ of right cosets is recurrent; as $C_\Gm(s)x \mapsto gC_\Gm(s)x = C_\Gm(gsg^{-1})gx$ is an isomorphism of $\Gm$-sets from $C_\Gm(s)\backslash\Gm$ to $C_\Gm(gsg^{-1})\backslash\Gm$, the conjugates of an element with a $\mu$-recurrent centraliser then also have $\mu$-recurrent centralisers.
So, if $\mu$ is non-degenerate, the subgroup $R_\mu(\Gm)$ generated by the elements with a $\mu$-recurrent centraliser is normal; by Lemma \ref{lemmart}, it contains the FC-centre.
As for the FC-central series, this construction can be iterated, since the push-forward of $\mu$ to a quotient of $\Gm$ is still non-degenerate.
\begin{dfn}\label{dfnRmu}
Let $\mu$ be a non-degenerate probability measure on a countable group $\Gm$.
The normal subgroups $R^\lambda_\mu$ of $\Gm$, indexed by the ordinals, are defined by transfinite induction: $R^0_\mu := \{e\}$, and, for every ordinal $\lambda$,
\[
R^{\lambda+1}_\mu := \phi_\lambda^{-1} \Big( R_{(\phi_\lambda)_*\mu}\big( \Gm/R^\lambda_\mu \big) \Big),
\]
where $\phi_\lambda : \Gm \surj \Gm/R^\lambda_\mu$ is the quotient map; finally $R^\beta_\mu := \bigcup_{\lambda<\beta} R^\lambda_\mu$ if $\beta$ is a limit ordinal.
This non-decreasing family of normal subgroups is constant from some ordinal on, and its final term, denoted by $\widehat{R}_\mu$, will be called the {\bfseries $\mu$-hyperrecurrent subgroup} of $\Gm$.
\end{dfn}
Note that $R^1_\mu = R_\mu(\Gm)$.
As the image of the FC-centre of a group in a quotient is contained in the FC-centre of the quotient, a transfinite induction shows that $Z^\lambda_{FC}(\Gm) \subset R^\lambda_\mu$ for every ordinal $\lambda$; in particular, $\widehat{R}_\mu$ contains the FC-hypercentre of $\Gm$.
If $N \lhd \Gm$ is contained in $\widehat{R}_\mu$, we shall say that $\Gm$ is a {\bfseries $\mu$-hyperrecurrent extension} of $\Gm/N$; for instance, extensions by normal subgroups contained in the FC-hypercentre are $\mu$-hyperrecurrent for every non-degenerate measure $\mu$.
\begin{teo}\label{teorec}
Let $\mu$ be a non-degenerate probability measure on a countable group $\Gm$, let $N \lhd \Gm$ be a normal subgroup contained in $\widehat{R}_\mu$ (for instance, a normal subgroup generated by elements with a $\mu$-recurrent centraliser, such as $R_\mu(\Gm)$), and let $\phi : \Gm \to \Gm/N$ be the quotient map.
Then every bounded $\mu$-harmonic function is constant on the $N$-cosets, so that $\BH_\mu(\Gm) = \{ \bar f \circ \phi \mid \bar f \in \BH_{\phi_*\mu}(\Gm/N) \}$.
In particular, $\Gm$ is $\mu$-Liouville if and only if $\Gm/N$ is $\phi_*\mu$-Liouville.
\end{teo}
\begin{proof}
Let $f$ be a bounded $\mu$-harmonic function.
If $s$ has a $\mu$-recurrent centraliser, then $s \in \spt(\mu^k)$ for some $k$, since $\mu$ is non-degenerate, so that $f(xs) = f(x)$ for all $x \in \Gm$ by Lemma \ref{lemmart}.
Hence $f$ is invariant under right multiplication by $R_\mu(\Gm)$.

Let us show, by transfinite induction on $\lambda$, that $f$ is constant on the $R^\lambda_\mu$-cosets; the case $\lambda = 0$ and the case of limit ordinals are as in the proof of Theorem \ref{teohyper}.
If $f = \bar f_\lambda \circ \phi_\lambda$ is constant on the $R^\lambda_\mu$-cosets, where $\phi_\lambda : \Gm \to \Gm/R^\lambda_\mu$ is the quotient map, then $\bar f_\lambda$ is a bounded $(\phi_\lambda)_*\mu$-harmonic function on $\Gm/R^\lambda_\mu$, and $(\phi_\lambda)_*\mu$ is non-degenerate.
By the previous paragraph, applied in $\Gm/R^\lambda_\mu$, $\bar f_\lambda$ is invariant under right multiplication by $R_{(\phi_\lambda)_*\mu}(\Gm/R^\lambda_\mu)$, \ie $f$ is constant on the $R^{\lambda+1}_\mu$-cosets.

Hence $f$ is constant on the cosets of $\widehat{R}_\mu$, and a fortiori on the $N$-cosets.
The rest of the statement follows as in the proof of Theorem \ref{teonondeg}.
\end{proof}
As $\widehat{R}_\mu$ contains the FC-hypercentre, Theorem \ref{teorec} contains Theorem \ref{teonondeg}.

Whether an element has a $\mu$-recurrent centraliser only depends on the random walks induced by $\mu$ on the sets $C_\Gm(s)\backslash\Gm$.
For symmetric measures, these random walks are reversible, and they can be compared by means of their Dirichlet forms.
\begin{lem}\label{lemdir}
Let $H < \Gm$ be a subgroup and $X := H\backslash\Gm$.
For a symmetric probability measure $\eta$ on $\Gm$ and $\varphi \in \ell^2(X)$, let
\[
\mathcal{E}_\eta(\varphi) := \tfrac{1}{2} \sum_{x \in X} \sum_{\gm \in \Gm} \eta(\gm)\, |\varphi(x\gm^{-1}) - \varphi(x)|^2.
\]
\begin{itemize}
\item[(a)] Let $\eta$ and $\eta'$ be symmetric probability measures on $\Gm$ whose supports generate $\Gm$, and assume that $\mathcal{E}_{\eta'} \leq C\, \mathcal{E}_\eta$ for some constant $C$.
If $H$ is $\eta$-recurrent, then $H$ is $\eta'$-recurrent.
\item[(b)] Let $\mu$ be a symmetric probability measure on $\Gm$ whose support $S$ is finite and generates $\Gm$, and let $|\gm|$ denote the word length of $\gm \in \Gm$ with respect to $S$.
Then, for all $\gm \in \Gm$ and $\varphi \in \ell^2(X)$,
\[\eqtag \label{eqpath}
\sum_{x \in X} |\varphi(x\gm) - \varphi(x)|^2 \leq \frac{2|\gm|^2}{\min_S \mu}\, \mathcal{E}_\mu(\varphi).
\]
Consequently, $\mathcal{E}_\nu \leq \frac{1}{\min_S \mu} \big( \sum_{\gm \in \Gm} \nu(\gm)\, |\gm|^2 \big)\, \mathcal{E}_\mu$ for every symmetric probability measure $\nu$ on $\Gm$.
\end{itemize}
\end{lem}
\begin{proof}
(a) As right multiplications are bijections of $X$ and $\eta$ is symmetric, the transition probabilities $p_\eta(x,y) := \eta\big(\{\gm \in \Gm \mid x\gm^{-1} = y\}\big)$ of the random walk induced by $\eta$ on $X$ satisfy $p_\eta(x,y) = p_\eta(y,x)$.
So this random walk is reversible with respect to the counting measure on $X$, and its Dirichlet norm, in the sense of Woess \cite[\S{}2]{Woe}, is $\mathcal{E}_\eta$.
As $\eta$ and $\eta'$ are non-degenerate, $H$ is $\eta$-recurrent, resp. $\eta'$-recurrent, if and only if the random walk induced by $\eta$, resp. $\eta'$, on $X$ is recurrent.
Since $\mathcal{E}_\eta \geq C^{-1} \mathcal{E}_{\eta'}$, Woess \cite[Corollary 2.14]{Woe} shows that, if the random walk induced by $\eta'$ on $X$ is transient, then so is the random walk induced by $\eta$.
This is the contrapositive of (a).

(b) Write $\gm = y_1 \cdots y_L$ with $L := |\gm|$ and $y_1, \ldots, y_L \in S$.
For every $x \in X$, the Cauchy--Schwarz inequality gives
\[
|\varphi(x\gm) - \varphi(x)|^2 \leq L \sum_{i=1}^{L} |\varphi(x y_1 \cdots y_i) - \varphi(x y_1 \cdots y_{i-1})|^2.
\]
For each $i$, the map $x \mapsto x y_1 \cdots y_{i-1}$ is a bijection of $X$, and $y_i^{-1} \in S$, as $S$ is symmetric; as $\mu(y_i^{-1}) \sum_{x} |\varphi(x y_i) - \varphi(x)|^2$ is one of the terms of the sum defining $2\,\mathcal{E}_\mu(\varphi)$, this gives
\[
\sum_{x \in X} |\varphi(x y_1 \cdots y_i) - \varphi(x y_1 \cdots y_{i-1})|^2 = \sum_{x \in X} |\varphi(x y_i) - \varphi(x)|^2 \leq \frac{2\,\mathcal{E}_\mu(\varphi)}{\min_S \mu}.
\]
Summing over $x \in X$ and $i \in \{1, \ldots, L\}$ gives \eqref{eqpath}.
Finally, as $|\gm^{-1}| = |\gm|$, \eqref{eqpath} gives
\[
\mathcal{E}_\nu(\varphi) = \tfrac{1}{2} \sum_{\gm \in \Gm} \nu(\gm) \sum_{x \in X} |\varphi(x\gm^{-1}) - \varphi(x)|^2 \leq \frac{1}{\min_S \mu} \sum_{\gm \in \Gm} \nu(\gm)\, |\gm|^2\, \mathcal{E}_\mu(\varphi). \qedhere
\]
\end{proof}

\begin{prop}\label{propindep}
Let $\mu$ be a symmetric probability measure on a countable group $\Gm$ whose support is finite and generates $\Gm$.
Let $\nu$ be a symmetric probability measure on $\Gm$ whose support generates $\Gm$, and which has a finite second moment, \ie $\sum_{\gm} \nu(\gm)\, |\gm|^2 < \infty$, where $|\cdot|$ is the word length with respect to some finite generating set of $\Gm$.
\begin{itemize}
\item[(a)] Every element of $\Gm$ with a $\mu$-recurrent centraliser has a $\nu$-recurrent centraliser, and $R^\lambda_\mu \subset R^\lambda_\nu$ for every ordinal $\lambda$.
\item[(b)] Every bounded $\nu$-harmonic function on $\Gm$ is constant on the cosets of $\widehat{R}_\mu$; in particular, $\Gm$ is $\nu$-Liouville if and only if $\Gm/\widehat{R}_\mu$ is Liouville for the push-forward of $\nu$.
\item[(c)] Whether an element of $\Gm$ has a $\mu$-recurrent centraliser, as well as the subgroups $R^\lambda_\mu$, do not depend on the choice of the symmetric finitely supported measure $\mu$ whose support generates $\Gm$.
\end{itemize}
\end{prop}
\begin{proof}
As word lengths with respect to finite generating sets are bi-Lipschitz equivalent, the finiteness of the second moment of $\nu$ does not depend on the choice of the word length; let $|\cdot|$ be the word length with respect to $\spt \mu$.

(a) If $s$ has a $\mu$-recurrent centraliser, Lemma \ref{lemdir}(b), applied to $X := C_\Gm(s)\backslash\Gm$, gives $\mathcal{E}_\nu \leq C\, \mathcal{E}_\mu$, where $C := \frac{1}{\min_{\spt \mu} \mu} \sum_\gm \nu(\gm)\, |\gm|^2 < \infty$, so that $s$ has a $\nu$-recurrent centraliser, by Lemma \ref{lemdir}(a).

Let us now show, by transfinite induction on $\lambda$, that $R^\lambda_\mu \subset R^\lambda_\nu$; the case $\lambda = 0$ and the case of limit ordinals are clear.
Assume that $R^\lambda_\mu \subset R^\lambda_\nu$, let $\phi : \Gm \to Q := \Gm/R^\lambda_\mu$ and $\phi' : \Gm \to Q' := \Gm/R^\lambda_\nu$ be the quotient maps, and let $\chi : Q \to Q'$ be the homomorphism such that $\phi' = \chi \circ \phi$.
The measures $\phi_*\mu$ and $\phi_*\nu$ on $Q$ satisfy the same assumptions as $\mu$ and $\nu$, since the word length of $\phi(\gm)$ with respect to $\phi(\spt \mu)$ is at most $|\gm|$.
So, by the previous paragraph, if $\bar s \in Q$ has a $\phi_*\mu$-recurrent centraliser, it has a $\phi_*\nu$-recurrent centraliser.
Moreover, $\phi^{-1}\big(C_Q(\bar s)\big) \subset \phi'^{-1}\big(C_{Q'}(\chi(\bar s))\big)$, so that $(\phi'_*\nu)^n\big(C_{Q'}(\chi(\bar s))\big) \geq (\phi_*\nu)^n\big(C_Q(\bar s)\big)$ for every $n$; hence $\chi(\bar s)$ has a $\phi'_*\nu$-recurrent centraliser.
Therefore $\chi$ maps $R_{\phi_*\mu}(Q)$ into $R_{\phi'_*\nu}(Q')$, \ie $R^{\lambda+1}_\mu \subset R^{\lambda+1}_\nu$.

(b) As $\nu$ is non-degenerate and $\widehat{R}_\mu \subset \widehat{R}_\nu$ by (a), this follows from Theorem \ref{teorec}, applied to $\nu$ and to $N := \widehat{R}_\mu$.

(c) If $\mu'$ is another symmetric finitely supported measure whose support generates $\Gm$, then (a) applies to the pairs $(\mu, \mu')$ and $(\mu', \mu)$, as finitely supported measures have a finite second moment.
\end{proof}
In particular, by Theorem \ref{teorec} and Proposition \ref{propindep}, the stability problem for $\Gm$ (see Section \ref{s-gen}) reduces to the same problem for $\Gm/\widehat{R}_\mu$: $\Gm$ is Liouville for every (resp. for no) symmetric finitely supported measure whose support generates it if and only if $\Gm/\widehat{R}_\mu$ is.
For instance, if $\Gm/\widehat{R}_\mu$ has subexponential growth, then $\Gm$ is $\nu$-Liouville for every symmetric measure $\nu$ with finite second moment whose support generates $\Gm$, since a group of subexponential growth is Liouville for every measure with finite first moment (see Avez \cite{Avez74} and Kaimanovich \& Vershik \cite{KV}).

Some moment condition is needed in Proposition \ref{propindep}.
For instance, $R^2_\mu = \Gm$ for the lamplighter group $\Gm := \zz/2\zz \wr \zz$ (see Example \ref{exalamp}), whereas $\Gm$, which is finitely generated but not virtually nilpotent, carries a symmetric measure $\nu$ whose support generates it and for which it is not Liouville (Frisch, Hartman, Tamuz \& Vahidi Ferdowsi \cite{FHTF}); by Theorem \ref{teorec}, $\widehat{R}_\nu \neq \Gm$ for such a measure $\nu$.

\section{Examples and applications}\label{s-exa}

In the following examples, $\mu$ is any symmetric finitely supported measure whose support generates the group $\Gm$ under consideration, and Theorem \ref{teorec} is applied with a quotient which is Abelian, or at least Liouville.
By Proposition \ref{propindep}, the description of the bounded harmonic functions which is obtained also holds for every symmetric measure with finite second moment whose support generates $\Gm$, and, by Theorem \ref{teorec}, for every non-degenerate measure for which the relevant induced random walks are recurrent.

\begin{exa}[Lamplighter groups]\label{exalamp}
Let $\Gm := L \wr \zz^d$, with $L$ countable Abelian and $d \leq 2$.
The centraliser of a non-trivial element of the lamp group $N := \bigoplus_{\zz^d} L$ is $N$ itself, so the induced random walk on its cosets is the projection to $\zz^d$ of the random walk, which is recurrent.
Hence $N \subset R_\mu(\Gm)$, every bounded $\mu$-harmonic function is lifted from $\zz^d$, and $\Gm$ is $\mu$-Liouville (see Kaimanovich \& Vershik \cite{KV}).
More generally, for a permutational wreath product $L \wr_X Q$ with $L$ Abelian, the lamps at a single point of $X$ have $\mu$-recurrent centralisers as soon as the random walk induced on $X$ is recurrent, and then every bounded $\mu$-harmonic function is lifted from $Q$.
\end{exa}

\begin{exa}[Metabelian groups]\label{exametab}
Let $\Gm := A \rtimes \zz^k$, with $A$ Abelian and $k \leq 2$, and assume that no non-trivial element of $\zz^k$ fixes a non-trivial element of $A$.
Then every non-trivial $a \in A$ has centraliser $A$, and $C_\Gm(a)\backslash\Gm \cong \zz^k$, so that $A \subset R_\mu(\Gm)$ and $\Gm$ is $\mu$-Liouville.
This covers the Baumslag--Solitar groups $BS(1,n) = \zz[\tfrac{1}{n}] \rtimes \zz$, the polycyclic groups $\zz^2 \rtimes_M \zz$ with $M$ hyperbolic, the groups $\zz \wr \zz$ and $\zz \wr \zz^2$, and the groups $\zz[\tfrac{1}{m}] \rtimes \zz^2$ where $(a,b) \in \zz^2$ acts by multiplication by $p^a q^b$, for any two multiplicatively independent integers $p,q$ whose prime factors divide $m$.
The extension need not split: the derived subgroup of the free metabelian group of rank 2 is a free cyclic module over $\zz[\zz^2]$, so that the same argument applies to this group.
\end{exa}

\begin{exa}[Houghton's groups and lampshufflers]\label{exahough}
For $n \geq 2$, let $H_n$ be Houghton's group, \ie the group of permutations of $X := \{1, \ldots, n\} \times \mathbb{N}$ which are eventually translations on each ray $\{i\} \times \mathbb{N}$ (so that $H_2 \cong \mathrm{FSym}(\zz) \rtimes \zz$ is the ``lampshuffler'' over $\zz$).
The finitary permutations form a normal subgroup $\mathrm{FSym}(X)$, generated by transpositions, with quotient $\zz^{n-1}$.
The conjugacy class of a transposition can be identified with the set of pairs of points of $X$.
Endow $X$ with the metric $d$ such that $d\big((i,k),(j,l)\big) = |k-l|$ if $i = j$, and $k+l+1$ otherwise; balls of radius $r$ for $d$ have at most $n(2r+1)$ points.
As every element of $H_n$ moves every point of $X$ by a bounded distance, the Schreier graph of the pairs of points of $X$ (with respect to $\spt \mu$) has quadratic volume growth, and is therefore recurrent, by Woess \cite[Lemma 3.12]{Woe}.
So transpositions have $\mu$-recurrent centralisers, and $H_n$ is $\mu$-Liouville for every $n$, although $\zz^{n-1}$ is transient for $n \geq 4$.

The lampshuffler $\mathrm{FSym}(\zz^2) \rtimes \zz^2$ can be treated in the same way, although the Schreier graph of pairs of points of $\zz^2$ now has growth of degree 4.
Let $r$ be such that every element of $\spt \mu$ acts as a translation on the points of $\zz^2$ at distance more than $r$ from the origin.
Then the difference $b-a$ of a pair $\{a,b\}$ is only changed by the steps of the random walk on pairs at which $a$ or $b$ is at distance at most $r$ from the origin, and in between, $a$ performs a recurrent random walk on $\zz^2$.
So the random walk on pairs, induced on the set $W$ of pairs having a point at distance at most $r$ from the origin, is well defined; it is reversible (with respect to the counting measure), and it changes $b-a$ by at most some constant $M$ at each step.
As the number of pairs in $W$ with $\|b-a\| \in [k,k+1[$ is $O(k)$, the Nash-Williams criterion (Woess \cite[Corollary 2.20]{Woe}, applied to the partition of $W$ according to the integer part of $\|b-a\|/M$) shows that this induced random walk is recurrent, hence so is the random walk on pairs.
Hence transpositions also have $\mu$-recurrent centralisers in this group, and $\mathrm{FSym}(\zz^2) \rtimes \zz^2$ is $\mu$-Liouville; this was first proved by Silva \cite{Sil}, with a different method (see Remark \ref{rmkknown}).
\end{exa}

\begin{exa}[Iterated wreath products]\label{exaiter}
In this example, $\widehat{R}_\mu$ is much larger than the FC-hypercentre and than $R^1_\mu$.
Let $L := \zz/2\zz \wr \zz$, with lamp group $N' := \bigoplus_\zz \zz/2\zz$, and let $\Gm := L \wr \zz$, with lamp group $N := \bigoplus_\zz L$.
The group $\Gm$ has infinite conjugacy classes, so that its FC-hypercentre is trivial.
Let $s \in N$ be supported at $0$, with $s(0) \in N' \setminus \{e\}$; its centraliser is $\{\psi \in N \mid \psi(0) \in N'\}$, so that $C_\Gm(s)\backslash\Gm$ can be identified with $\zz \times \zz$.
The first coordinate is the position in the (outer) base $\zz$, and the second one, the image in $L/N' \cong \zz$ of the lamp at $0$, only changes when the first one is close to $0$.
The corresponding Schreier graph is thus a comb, which has quadratic growth and is therefore recurrent (Woess \cite[Lemma 3.12]{Woe}).
Hence $R^1_\mu \supset \bigoplus_\zz N'$; as $\Gm/\bigoplus_\zz N' \cong \zz \wr \zz$, Example \ref{exametab} gives $R^2_\mu \supset N$, and $R^3_\mu = \Gm$.
(One can check that the other elements have transient centralisers, so that $R^1_\mu = \bigoplus_\zz N'$ and $R^2_\mu = N$.)
Iterating this construction, \ie considering $(\cdots((\zz/2\zz \wr \zz) \wr \zz) \cdots) \wr \zz$, one gets series $(R^\lambda_\mu)_\lambda$ of arbitrarily large finite length.
\end{exa}

\begin{exa}[Permutational wreath products over fractal Schreier graphs]\label{exafrac}
Let $G$ be a finitely generated group of automorphisms of bounded type of a rooted tree, such as Grigorchuk's group, the Basilica group, the Hanoi towers group or, more generally, any group generated by a bounded automaton.
Its orbital Schreier graphs on the boundary of the tree are recurrent (Bondarenko \cite{Bon}, see \cite[Proposition 3.7]{AAMV}), and Amir, Angel, Matte Bon \& Vir\'ag \cite{AAMV} showed that $G$ is Liouville for every symmetric finitely supported measure.
So, if $X$ is an orbit of $G$ on the boundary of the tree and $L$ is a countable Abelian group, Example \ref{exalamp} shows that $L \wr_X G$ is Liouville for every symmetric finitely supported measure whose support generates it.
When $G$ is the Basilica group, this group has exponential growth.
\end{exa}

\begin{rmk}\label{rmklimits}
The criterion of course fails when the relevant random walks are transient, \eg for lamplighter groups over $\zz^d$ with $d \geq 3$, which are not Liouville (see Kaimanovich \& Vershik \cite{KV}).
Note also that it depends on the size of the conjugacy classes, and not on the length of a series of normal subgroups.
For instance, let $\mathcal{O}$ be the ring of integers of a totally real number field of degree $4$, and let $U \cong \zz^3$ be a torsion-free subgroup of finite index of the group of units of $\mathcal{O}$, acting on $\mathcal{O} \cong \zz^4$ by multiplication.
The polycyclic group $\Gm := \mathcal{O} \rtimes U$ is an iterated extension of cyclic groups, but $R_\mu(\Gm)$, and hence $\widehat{R}_\mu$, is trivial.
Indeed, the centraliser of a non-trivial $a \in \mathcal{O}$ is $\mathcal{O}$, as no unit $u \neq 1$ fixes $a$, so that $C_\Gm(a)\backslash\Gm \cong U$, on which the induced random walk is transient.
If $u \in U \setminus \{1\}$ and $a \in \mathcal{O}$, then $(c,v) \in \Gm$ commutes with $(a,u)$ if and only if $va + (1-u)c = a$; as $1-u \neq 0$, $c$ is determined by $v$, so that the centraliser $C_{a,u}$ of $(a,u)$ projects injectively to $U \cong \zz^3$.
As this projection is Lipschitz, $C_{a,u}$ contains at most $Kn^3$ elements of $\spt(\mu^n)$, for some constant $K$.
On the other hand, as $\Gm$ has exponential growth, $\mu^{2n}(e) \leq \exp(-cn^{1/3})$ for some $c > 0$ and all $n$ large enough (Varopoulos, see \eg Pittet \& Saloff-Coste \cite{PSC}), and $\mu^n(y)^2 \leq \sum_z \mu^n(z)^2 = \mu^{2n}(e)$ for all $y \in \Gm$, as $\mu$ is symmetric.
So $\mu^n(C_{a,u}) \leq Kn^3 \exp(-cn^{1/3}/2)$ for $n$ large enough, and $C_{a,u}$ is not $\mu$-recurrent.
These groups are nevertheless $\mu$-Liouville (see Remark \ref{rmkknown}).
In the iterated wreath products of Example \ref{exaiter}, the conjugacy class of an inner lamp also has several coordinates, but each of them only moves when the previous ones are at the origin, and this preserves recurrence.

In the other direction, Theorem \ref{teorec} shows that, if $(\Gm, \mu)$ is not Liouville, then $\widehat{R}_\mu$ acts trivially on its Poisson boundary, and in particular $\widehat{R}_\mu \neq \Gm$.
For instance, every non-trivial normal subgroup of Thompson's group $F$ contains $[F,F]$, and $F/[F,F] \cong \zz^2$; since $F$ is not $\mu$-Liouville (see Kaimanovich \cite{Kai-F}), no non-trivial element of $F$ has a $\mu$-recurrent centraliser.
Similarly, Brieussel, Tanaka \& Zheng \cite{BTZ} showed that the discrete affine group of a regular tree (the group of automorphisms of the tree which fix an end and have a finitely supported portrait) is not $\mu$-Liouville, and their proof also applies to its variants with labels in a transitive group of permutations \cite[Remark 3.4]{BTZ}, such as the ascending HNN extension $\big(\cdots \wr C_p \wr C_p\big) \rtimes \zz$ of the group of finitary automorphisms of the rooted $p$-ary tree with labels in $C_p$.
In these groups, the centraliser of an automorphism whose portrait is supported at a single vertex $v$ fixes $v$, and the random walk induced on the vertices of the tree is transient \cite[Lemma 4.5]{BTZ}; this transience is also what makes the portraits stabilise along the random walk.
\end{rmk}

\begin{rmk}\label{rmkknown}
For several of the groups of Examples \ref{exalamp}--\ref{exafrac}, the Liouville property for all symmetric finitely supported measures whose support generates the group was already known, by various methods.
\begin{itemize}
\item Let $L \wr B$ be a wreath product, where $L$ is non-trivial and $B$ is infinite and finitely generated.
Erschler \cite[Theorem 3.1]{Ers-LG} showed that $L \wr B$ is not $\mu$-Liouville for any non-degenerate measure $\mu$ with finite entropy whose projection to $B$ is transient.
Conversely, Lyons \& Peres \cite[Proposition 4.9]{LP} showed that, if $L$ is FC-hypercentral, then $L \wr B$ is $\mu$-Liouville for every measure $\mu$ whose support generates $L \wr B$ and whose projection to $B$ is recurrent (see Kaimanovich \& Vershik \cite[Proposition 6.3]{KV} and Kaimanovich \cite{Kai91} for earlier results).
By a theorem of Varopoulos (see \eg Woess \cite[Theorem 3.24]{Woe}), the symmetric finitely supported random walks on $B$ are recurrent if and only if $B$ is virtually $\zz$ or $\zz^2$; so this settles the stability problem for these wreath products.
These results do not apply directly to iterated wreath products such as $(\zz/2\zz \wr \zz) \wr \zz$, whose lamp group $\zz/2\zz \wr \zz$ is not FC-hypercentral; however, Erschler \cite{Ers-drift} showed that the drift of some random walks on iterated wreath products of $\zz$ is sublinear, so that these random walks are Liouville.
\item Pittet \& Saloff-Coste \cite{PSC} showed that, on a solvable group of finite Pr\"ufer rank, such as a polycyclic group, $BS(1,n)$ or $\zz[\tfrac{1}{m}] \rtimes \zz^k$, the probability of return of a symmetric finitely supported random walk at time $2n$ is at least $\exp(-Cn^{1/3})$.
By a theorem of Peres \& Zheng \cite{PZ}, according to which such a slow decay of the probability of return implies the Liouville property, these groups are $\mu$-Liouville for every symmetric measure $\mu$ with finite first moment whose support generates them.
\item Frisch \& Silva \cite[Corollary 1.9]{FS} showed that, if $\mu$ is a measure with finite first moment whose support generates the free metabelian group of rank 2, then this group is $\mu$-Liouville if and only if the projection of $\mu$ to $\zz^2$ is recurrent.
\item Silva \cite[Propositions 6.1 and 6.2]{Sil} showed that, for $d \in \{1,2\}$, the lampshuffler $\mathrm{FSym}(\zz^d) \rtimes \zz^d$ is $\mu$-Liouville for every finitely supported measure $\mu$ whose projection to $\zz^d$ is recurrent.
\end{itemize}
We are not aware of earlier results for Houghton's groups $H_n$ with $n \geq 3$ (Example \ref{exahough}), nor for the groups $L \wr_X G$ of Example \ref{exafrac}.
Moreover, as Theorem \ref{teorec} requires no moment condition (see the beginning of this section), if $L$ is Abelian and $d \leq 2$, then $L \wr \zz^d$ is $\mu$-Liouville for every non-degenerate measure $\mu$ whose projection to $\zz^d$ is recurrent (as in \cite[Proposition 4.9]{LP}), and the same holds for the free metabelian group of rank 2.
\end{rmk}

\begin{rmk}\label{rmkinner}
Let $\mu$ be a non-degenerate probability measure on $\Gm$.
For $s \in \Gm$, the map $C_\Gm(s)x \mapsto x^{-1}sx$ identifies $C_\Gm(s)\backslash\Gm$ with the conjugacy class of $s$, and the random walk induced by $\mu$ on $C_\Gm(s)\backslash\Gm$ with the random walk $(X_n^{-1}sX_n)_n$, whose steps $y \mapsto YyY^{-1}$ (with $Y$ of law $\mu$) are inner automorphisms.
Let us say that $(\Gm,\mu)$ is {\bfseries inner recurrent} if every element of $\Gm$ has a $\mu$-recurrent centraliser, \ie if all these random walks are recurrent, and that $\Gm$ is {\bfseries inner $\mu$-Liouville} if none of these random walks has non-constant bounded harmonic functions.
Equivalently, $\Gm$ is inner $\mu$-Liouville if every bounded $\mu$-harmonic function on $\Gm$ which is invariant under left multiplication by the centraliser of some element of $\Gm$ is constant, \ie if the centraliser of every element of $\Gm$ acts ergodically on the Poisson boundary of $(\Gm,\mu)$.
Consider the following properties:
\begin{itemize}
\item[(i)] $(\Gm,\mu)$ is recurrent, \ie the trivial subgroup is $\mu$-recurrent;
\item[(ii)] $(\Gm,\mu)$ is inner recurrent;
\item[(iii)] $\widehat{R}_\mu = \Gm$;
\item[(iv)] $\Gm$ is $\mu$-Liouville;
\item[(v)] $\Gm$ is inner $\mu$-Liouville.
\end{itemize}
Then (i) $\imp$ (ii) $\imp$ (iii) $\imp$ (iv) $\imp$ (v): the first implication holds since $e \in C_\Gm(s)$ for every $s$, the second one since (ii) implies $R^1_\mu = \Gm$, the third one is Theorem \ref{teorec}, and the last one is obvious.
For symmetric finitely supported measures $\mu$ whose support generates $\Gm$, the first three implications are strict.
Indeed, $\zz^3$, like every Abelian group, is inner recurrent, but it is not recurrent.
For the lamplighter group $\Gm := \zz/2\zz \wr \zz$, one has $R^2_\mu = \Gm$ (see Example \ref{exalamp}), but the centraliser of the generator $t$ of the base $\zz$ is $\langle t \rangle$, and $t$ does not have a $\mu$-recurrent centraliser.
For the switch-walk-switch measure, this follows from the fact that $\mu^n(\langle t \rangle)$, the probability that all lamps are off at time $n$, is $\mathbb{E}[2^{-\mathcal{R}_n}]$, where $\mathcal{R}_n$ is the number of points visited up to time $n$ by a simple random walk on $\zz$, and is thus at most $\exp(-cn^{1/3})$ for some $c > 0$ and all $n$ large enough; the general case follows by Proposition \ref{propindep}(c).
Finally, the groups $\mathcal{O} \rtimes U$ of Remark \ref{rmklimits} are $\mu$-Liouville, but $\widehat{R}_\mu$ is trivial.

On the other hand, the free group $F_2$ is not inner $\mu$-Liouville: the centraliser of a non-trivial element is cyclic, and its action on the boundary of the tree, endowed with the harmonic measure, has a wandering set of positive measure.
This leads to the following question.
\end{rmk}

\begin{ques}\label{quesinner}
Let $\mu$ be a symmetric finitely supported probability measure whose support generates a countable group $\Gm$.
If $\Gm$ is inner $\mu$-Liouville, is it $\mu$-Liouville?
\end{ques}

Here is an amusing corollary in the context of \cite{Go-mixing}.
\begin{cor}\label{corFCcoc}
Let $\mu$ be a non-degenerate probability measure on a countable group $\Gm$ such that $\widehat{R}_\mu$ is infinite (\eg $\Gm$ has an infinite FC-hypercentre), and let $\pi: \Gm \to \mathrm{Aut}(\mathsf{H})$ be a Hilbert representation with finite stabilisers.
Then any non-trivial $\mu$-harmonic cocycle $b \in \srl{H}^1(\Gm,\pi)$ is unbounded.
\end{cor}
\begin{proof}
If $b$ is $\mu$-harmonic and bounded, then it yields bounded $\mu$-harmonic functions on $\Gm$.
By Theorem \ref{teorec}, these functions are constant on the cosets of $\widehat{R}_\mu$, which means that $b$ is trivial on $\widehat{R}_\mu$, as $b(e_\Gm)=0$.
But by \cite[Lemma 2.11]{Go-mixing} this extends to the almost-malnormal hull of $\widehat{R}_\mu$, which is $\Gm$ since $\widehat{R}_\mu$ is an infinite normal subgroup.
\end{proof}
For instance, this applies to the groups of Examples \ref{exalamp}--\ref{exafrac}, for every symmetric measure with finite second moment whose support generates them (by Proposition \ref{propindep}).

See Section \ref{s-gen} (Propositions \ref{proprec} and \ref{proprecP} and Corollary \ref{correc}) for consequences of Lemma \ref{lemmart} on the choice of the measure.

\section{Remarks on generating sets}\label{s-gen}

Though there are many measures $\mu$ which fit in the description of Theorems \ref{teohyper} and \ref{teofchyper}, it is natural to wonder if the choice is important.
For symmetric finitely supported measures whose support generates the group, this is the stability problem recalled in the introduction (see Erschler \& Frisch \cite{EF} for more details); the finiteness of the supports, or at least some moment condition, matters here, as observed after Proposition \ref{propindep}.

Our remarks here will be mostly motivated by the following observation.
Let $P$ and $Q$ be two finitely supported measures whose supports generate $\Gm$, and assume that $P$ is symmetric and that the identity element $e$ belongs to $\spt P$.
Then there exists $n$ such that $\spt (P^n) \supset \spt Q$: indeed $\spt(P^n) = (\spt P)^n$ is non-decreasing in $n$ and exhausts $\Gm$.
Moreover $P$ and $P^n$ have the same bounded harmonic functions, by the following lemma.
\begin{lem}\label{lempow}
Let $\alpha$ be a probability measure on $\Gm$ such that $\alpha(e) > 0$.
Then $\BH_{\alpha^n}(\Gm) = \BH_{\alpha}(\Gm)$ for every $n \geq 1$.
\end{lem}
\begin{proof}
Clearly $\BH_{\alpha}(\Gm) \subset \BH_{\alpha^n}(\Gm)$.
Let $\mathcal{B}$ be the space of bounded $\cc$-valued $\alpha^n$-harmonic functions; it is a closed subspace of $\ell^\infty(\Gm;\cc)$.
Since $\alpha$ commutes with $\alpha^n$, the operator $T : h \mapsto h*\alpha$ preserves $\mathcal{B}$, and $T^n = I$ on $\mathcal{B}$.
Let $a := \alpha(e)$.
The operator $T - aI$ is the convolution by the positive measure $\alpha - a\delta_e$, whose mass is $1-a$, so that $\|T - aI\| \leq 1-a$.
If $\omega \neq 1$ is an $n$-th root of unity, then $|\omega - a|^2 - (1-a)^2 = 2a(1 - \mathrm{Re}\,\omega) > 0$; hence
$T - \omega I = (a - \omega)\big(I + (a-\omega)^{-1}(T - aI)\big)$ is invertible on $\mathcal{B}$ (by a formal [Neumann] series).
Since these operators commute and $\prod_{\omega^n = 1} (T - \omega I) = T^n - I = 0$ on $\mathcal{B}$, it follows that $T = I$ on $\mathcal{B}$, \ie every bounded $\alpha^n$-harmonic function is $\alpha$-harmonic.
\end{proof}
Both facts fail without the assumption $e \in \spt P$, as the example of the simple random walk on $\zz$ given in Section \ref{s-proj} shows.
Hence the stability problem reduces to the question whether, given a measure $Q$, the Liouville property passes to the convex combinations of $Q$ and of a Dirac mass $\delta_\gm$: successive such combinations lead from $Q$ to $P^n$.

Often in this section, the formal [or Neumann] series $(1-x)^{-1} = \sum_{i\geq 0} x^i$ [for operators of norm $<1$] will be used. There are of course other ways to come to the same conclusions.
\begin{prop}\label{propnu}
Assume $c$ is an element of order $n$ and $P$ is a probability measure on $\Gm$.
Assume $\mu_t = tP+(1-t)\delta_c$ for some $t \in ]0,1[$.
Let $\nu_t = \displaystyle P*\sum_{i=0}^{n-1} \tfrac{t(1-t)^i}{1-(1-t)^n} \delta_c^i$.
Then $f$ is $\mu_t$-harmonic if and only if $f$ is $\nu_t$-harmonic.
\end{prop}
\begin{proof}
Assume $f$ is $\mu_t$-harmonic.
Then $f*\mu_t = f = f*\delta_e$ can be rearranged as $f*P = f* (\tfrac{1}{t} \delta_e - \tfrac{1-t}{t}\delta_c)$.
Hence, by inverting $\delta_e - (1-t) \delta_c$, and noticing that, because of the finite order of $c$, $\displaystyle \sum_{i\geq 0} [(1-t)\delta_c]^i = \sum_{i=0}^{n-1} \tfrac{(1-t)^i}{1-(1-t)^n} \delta_c^i$, one gets
\[
f* P * ( \sum_{i=0}^{n-1} \tfrac{t(1-t)^i}{1-(1-t)^n} \delta_c^i) = f.
\]
That is, $f$ is $\nu_t$-harmonic.

Conversely, assume that $f$ is $\nu_t$-harmonic.
Then $f= f*\nu_t$ so $f*(\delta_e - (1-t)\delta_c) = f*tP$. This can be rearranged to see that $f$ is $\mu_t$-harmonic.
\end{proof}

For instance, if $t = \sfrac{1}{2}$ and $c$ has order 2, then $\nu_{\sfrac{1}{2}} = \tfrac{2}{3} P + \tfrac{1}{3} P*\delta_c$.

In Proposition \ref{propnu}, $\mu_t$ is not symmetric unless $c$ has order 2. But even if $c$ has order 2, $\nu_t$ is not symmetric, unless $\delta_c*P = P*\delta_c$ (but then Lemma \ref{lemprinc} can be applied).
The corresponding result for the symmetric measure $\mu'_t := tP + (1-t)\varsigma_c$, where $\varsigma_c := \tfrac{1}{2}(\delta_c + \delta_{c^{-1}})$, is obtained in the same way: a function is $\mu'_t$-harmonic if and only if it is harmonic with respect to $\nu'_t := tP * \sum_{i \geq 0} (1-t)^i \varsigma_c^i$.
The coefficients of $\nu'_t$ involve Čebyšëv polynomials.
Let $u := \tfrac{1}{1-t}$ and $\lambda := u - \sqrt{u^2-1} = \frac{1-\sqrt{t(2-t)}}{1-t} \in \, ]0,1[$.
If $c$ has infinite order, then
\[
\nu'_t = \sqrt{\frac{t}{2-t}}\, \sum_{k \in \zz} \lambda^{|k|}\, P*\delta_c^k,
\]
and, if $c$ has finite order $n \geq 2$, then
\begin{align*}
\nu'_t &= \frac{1-t}{2-t} \sum_{r=0}^{n-1} \frac{T_r(u) + T_{n-r}(u)}{U_{n-1}(u)}\, P*\delta_c^r \\
&= \sqrt{\frac{t}{2-t}}\, \sum_{r=0}^{n-1} \frac{\lambda^r + \lambda^{n-r}}{1-\lambda^n}\, P*\delta_c^r,
\end{align*}
where $T_m$ and $U_m$ are the Čebyšëv polynomials of the first and second kind, \ie $T_m(\cosh\theta) = \cosh(m\theta)$ and $\sinh\theta \cdot U_{m-1}(\cosh\theta) = \sinh(m\theta)$ for $\theta \in \rr$.
Indeed, $\varsigma_c^i = \sum_{k \in \zz} \pp(Z_i = k)\, \delta_c^k$, where $(Z_i)_i$ is the simple random walk on $\zz$ started at $0$, and, for $s := 1-t$ (so that $1-s^2 = t(2-t)$), the classical generating function
\[
\sum_{i \geq 0} s^i\, \pp(Z_i = k) = \frac{\lambda^{|k|}}{\sqrt{1-s^2}} \qquad (k \in \zz)
\]
gives the first formula.
For the second one, sum over $k \in r + n\zz$, and write $u = \cosh\theta$, so that $\lambda = e^{-\theta}$ and
\[
\sum_{q \in \zz} \lambda^{|r+qn|} = \frac{\lambda^r + \lambda^{n-r}}{1-\lambda^n} = \frac{\cosh\big((\tfrac{n}{2}-r)\theta\big)}{\sinh(\tfrac{n}{2}\theta)} \qquad (0 \leq r \leq n-1);
\]
moreover,
\begin{align*}
\sqrt{\frac{t}{2-t}} &= \tanh\frac{\theta}{2} = \frac{\sinh\theta}{1+\cosh\theta} = \frac{1-t}{2-t}\, \sinh\theta, \\
\frac{T_r(u) + T_{n-r}(u)}{U_{n-1}(u)} &= \frac{\sinh\theta\, \cosh\big((\tfrac{n}{2}-r)\theta\big)}{\sinh(\tfrac{n}{2}\theta)}.
\end{align*}
For $n = 2$, one recovers the measure $\nu_t = \tfrac{1}{2-t}P + \tfrac{1-t}{2-t}P*\delta_c$ of Proposition \ref{propnu}.
But although $\mu'_t$ is symmetric, $\nu'_t$ is not symmetric.
\begin{rmk}\label{remSWS}
Also noteworthy is that, if one lets $\rho = \tfrac{1}{n} \sum_{i=0}^{n-1} \delta_c^i$ be the uniform measure on the group generated by $c$, when $t=0$, $\mu_t$ is very singular (reduces to $\delta_c$), but the limit $\nu_0 := \lim_{t\to 0} \nu_t = P * \rho $ remains an interesting measure.
Furthermore, a $\nu_0$-harmonic function is also $\delta_c$-harmonic (since $\rho * \delta_c = \rho$).
But if $f$ is $\delta_c$-harmonic and $\nu_0$-harmonic, then it is harmonic with respect to the measures $\tfrac{1}{n} \sum_{i = 0}^{n-1} \delta_c^{-i} P \delta_c^i$ and $\rho*P*\rho$ as well.
These last two measures are symmetric (if $P$ is) and the latter measure is akin to a ``switch-walk-switch'' generating set.
Lastly, note that $ \nu_t * \rho = P * \rho$. Hence if $f$ is $\mu_t$-harmonic, then $ f*P*\rho = f* \rho$.
\end{rmk}

Another reformulation of $\mu_t$-harmonicity, valid for every $c \in \Gm$ (of finite order or not), is the following.
\begin{prop}\label{propS}
Let $P$ be a probability measure on $\Gm$.
Assume $\mu_t = tP+(1-t)\delta_c$ for some $t \in ]0,1[$ and $c \in \Gm$.
Let $S_t = (1-t)\sum_{i \geq 0} (tP)^i$.
Then $f$ is $\mu_t$-harmonic if and only if $f$ is $(\delta_c*S_t)$-harmonic.
Furthermore, if $c$ is an element of order $n$, then $f$ is $\Big(\prod_{i=1}^{n} \delta_c^{i}S_t\delta_c^{-i} \Big)$-harmonic.
\end{prop}
\begin{proof}
The fact that $f$ is $\mu_t$-harmonic can be written as $f*(1-t)\delta_c = f*\delta_e - f*tP$.
Convolving on the right with $(\delta_e-tP)^{-1} = \sum_{i \geq 0} (tP)^i$, one gets $f*\delta_c *S_t = f$.
This is equivalent to $\delta_c*S_t$-harmonicity.
Starting from $f = f*\delta_c*S_t$, use convolution by $(\delta_e - tP)$ to get back to $\mu_t$-harmonicity.

Furthermore, $f$ is then automatically $(\delta_c*S_t)^n$-harmonic.
But $\delta_c*S_t*\delta_c*S_t*\delta_c*S_t*\cdots *\delta_c*S_t$ can be rewritten as
$\delta_c*S_t*\delta_c^{-1} * \delta_c^2*S_t*\delta_c^{-2} * \delta_c^3*S_t*\delta_c^{-3} \cdots *\delta_c^n*S_t$. The last $\delta_c^{-n}$ can be added since $\delta_c^n = \delta_c^{-n} = \delta_e$.
\end{proof}
\begin{rmk}
The interesting property of the measure $\Big(\prod_{i=1}^{n} \delta_c^{i}S_t\delta_c^{-i} \Big)$ is that if $P$ is Liouville, then so are all the $\delta_c^i * S_t * \delta_c^{-i}$.
Thus one passes from a problem on sums of measures to a problem on products of measures.

To see this, first note that, assuming $P$ is Liouville, $\delta_c^i * P * \delta_c^{-i}$ is also Liouville, since conjugation is a group automorphism.
Then note that if $f$ is $S_t$-harmonic, then convolving $f = f*S_t$ on the right with $\delta_e - tP$ yields $f*(\delta_e - tP) = (1-t)f$. Rearranging, one sees that $t f = t f*P$, \ie $f$ is $P$-harmonic.

The converse, namely that any $P$-harmonic function is $S_t$-harmonic, is straightforward.
\end{rmk}
In fact, when $P$ is symmetric, nothing is lost by passing to products:
\begin{prop}\label{propequiv}
Let $c$, $n$, $t$, $\mu_t$ and $S_t$ be as in Proposition \ref{propS}, assume moreover that $P$ is symmetric and that its support generates $\Gm$, and let $\Pi_t := \prod_{i=1}^{n} \delta_c^{i}S_t\delta_c^{-i}$.
Then
\[
\BH_{\mu_t}(\Gm) = \BH_{\delta_c * S_t}(\Gm) = \BH_{\Pi_t}(\Gm).
\]
In particular, $\Gm$ is $\mu_t$-Liouville if and only if it is $\Pi_t$-Liouville.
\end{prop}
\begin{proof}
The first equality is Proposition \ref{propS}, and $(\delta_c * S_t)^n = \Pi_t$ was observed in its proof.
Since $\spt P$ is symmetric and generates $\Gm$, $c^{-1}$ belongs to $\spt(P^k)$ for some $k$, so that $(\delta_c * S_t)(e) = S_t(c^{-1}) \geq (1-t)t^k P^k(c^{-1}) > 0$ (in fact $S_t$, and hence $\delta_c * S_t$, has full support).
The second equality then follows from Lemma \ref{lempow} applied to $\alpha = \delta_c * S_t$.
\end{proof}
Hence, if $P$ is symmetric and Liouville, then $\Gm$ is $\mu_t$-Liouville if and only if the product $\Pi_t$ of the Liouville measures $\delta_c^{i} * S_t * \delta_c^{-i}$ is Liouville.
The measures in Proposition \ref{propS} are not symmetric in general; for the symmetric measure $\mu'_t = tP + (1-t)\varsigma_c$, the analogue of Proposition \ref{propS} reads as follows, the product of conjugates being replaced by an average of such products.
\begin{prop}\label{propSsym}
Let $P$ be a probability measure on $\Gm$, let $c \in \Gm$ and $t \in ]0,1[$, and let $S_t$ be as in Proposition \ref{propS}.
Then $f$ is $\mu'_t$-harmonic if and only if $f$ is $(\varsigma_c * S_t)$-harmonic.
Moreover, for every $k \geq 1$,
\[
(\varsigma_c * S_t)^k = \frac{1}{2^k} \sum_{\varepsilon \in \{\pm 1\}^k} \Big( \prod_{j=1}^{k} \delta_c^{\ell_j} S_t \delta_c^{-\ell_j} \Big) \delta_c^{\ell_k}, \qquad \textrm{where } \ell_j := \varepsilon_1 + \cdots + \varepsilon_j.
\]
\end{prop}
\begin{proof}
The equation $f*\mu'_t = f$ can be rewritten as $f*(\delta_e - tP) = (1-t) f*\varsigma_c$.
Convolving on the right with $(\delta_e - tP)^{-1} = \sum_{i \geq 0} (tP)^i$ gives $f = f*\varsigma_c*S_t$; conversely, convolving $f = f*\varsigma_c*S_t$ on the right with $\delta_e - tP$ gives back $f*(\delta_e - tP) = (1-t) f*\varsigma_c$.
For the formula, expand $(\varsigma_c*S_t)^k$ as $2^{-k} \sum_\varepsilon \delta_c^{\varepsilon_1} S_t \delta_c^{\varepsilon_2} S_t \cdots \delta_c^{\varepsilon_k} S_t$.
As in the proof of Proposition \ref{propS}, since $\varepsilon_j = \ell_j - \ell_{j-1}$ (with $\ell_0 := 0$), inserting $\delta_c^{-\ell_j} \delta_c^{\ell_j} = \delta_e$ after the $j$-th factor $S_t$ shows that each term is equal to $\big( \prod_{j=1}^{k} \delta_c^{\ell_j} S_t \delta_c^{-\ell_j} \big) \delta_c^{\ell_k}$.
\end{proof}
In other words, $(\varsigma_c * S_t)^k$ is an average, over the paths $(\ell_j)_{j \leq k}$ of the simple random walk on $\zz$, of products of conjugates of $S_t$ by powers of $c$, followed by $\delta_c^{\ell_k}$.
When $c$ has order 2, $\varsigma_c = \delta_c$ and one recovers Proposition \ref{propS}.
If moreover $P$ is symmetric and its support generates $\Gm$, then $S_t$ is symmetric and $(\varsigma_c * S_t)(e) = S_t(c) > 0$, so that, by Lemma \ref{lempow}, $\BH_{\mu'_t}(\Gm) = \BH_{\varsigma_c * S_t}(\Gm) = \BH_{(\varsigma_c * S_t)^k}(\Gm)$ for every $k \geq 1$, as in Proposition \ref{propequiv}.

We now come back to the question raised at the beginning of this section, namely whether the Liouville property passes from $P$ to $\mu_t = tP + (1-t)\delta_c$.
Lemma \ref{lemmart} already gives one inclusion between the bounded harmonic functions of $P$ and of $\mu_t$.
Namely, if $c \in \spt(P^k)$ for some $k$ and $c$ has a $P$-recurrent centraliser, then every bounded $P$-harmonic function $f$ satisfies $f*\delta_c = f$, hence $f*\mu_t = t f*P + (1-t) f*\delta_c = f$.
That is, $\BH_P(\Gm) \subset \BH_{\mu_t}(\Gm)$, and the latter is equal to $\BH_{\nu_t}(\Gm)$ by Proposition \ref{propnu} when $c$ has finite order.
The reverse inclusion is more delicate; here is a first situation in which it holds, again by Lemma \ref{lemmart}.
\begin{prop}\label{proprec}
Let $P$ be a probability measure on $\Gm$, let $c \in \Gm$, $t \in ]0,1[$ and $\mu_t := tP + (1-t)\delta_c$.
Assume that $c$ has a $\mu_t$-recurrent centraliser (this holds for instance if the conjugacy class of $c$ is finite).
Then $\BH_{\mu_t}(\Gm) = \{ f \in \BH_P(\Gm) \mid f*\delta_c = f \}$.
In particular, if $\Gm$ is $P$-Liouville, then it is $\mu_t$-Liouville.
\end{prop}
\begin{proof}
The inclusion $\supset$ is clear.
Conversely, let $f$ be a bounded $\mu_t$-harmonic function.
Since $c \in \spt \mu_t$, Lemma \ref{lemmart} shows that $f(xc) = f(x)$ for all $x \in \Gm$, \ie $f*\delta_c = f$.
Hence $f = f*\mu_t = t f*P + (1-t) f$, so that $f$ is $P$-harmonic.
\end{proof}
Proposition \ref{proprec} requires a recurrence assumption on the random walk driven by $\mu_t$, whereas the inclusion $\BH_P(\Gm) \subset \BH_{\mu_t}(\Gm)$ obtained before Proposition \ref{proprec} only involves the random walk driven by $P$.
We now show that the former assumption follows from the latter when $P$ is symmetric and finitely supported, provided that $c$ is added symmetrically (which is automatic if $c$ is an involution), or that $c$ has finite order.
In the latter case, the measure $\delta_c * S_t$ of Proposition \ref{propS} is more convenient than $\mu_t$.
For instance, the position after $k$ steps of the random walk driven by $\delta_c * S_t$ is of the form $\big(\prod_{i=1}^{k} U_i c^{-1} U_i^{-1}\big) U_k$, where $(U_i)_i$ is a random walk driven by $S_t$, so that it always involves exactly $k$ conjugates of $c^{-1}$.
More importantly, $S_t$ gives a positive operator, whereas neither $P$ nor $\mu_t$ does in general.
We shall use the following technical inequalities.
\begin{lem}\label{lemprodpos}
Let $A_1, \ldots, A_n$ be self-adjoint operators on a Hilbert space such that $0 \leq A_j \leq I$, and let $\mathcal{E}(f) := \sum_j \mathcal{E}_j(f)$, where $\mathcal{E}_j(f) := \langle (I - A_j) f \mid f \rangle$.
Then, for all vectors $f$ and $g$,
\[
\mathrm{Re}\, \langle (I - A_1 \cdots A_n) f \mid f \rangle \geq \tfrac{1}{2n^2} \mathcal{E}(f),
\]
\[
\big| \langle (I - A_1 \cdots A_n) f \mid g \rangle \big| \leq 2^{n+1} \sqrt{n}\, \mathcal{E}(f)^{1/2} \mathcal{E}(g)^{1/2}.
\]
\end{lem}
\begin{proof}
In this proof, TI and CS refer to the triangle inequality and to the Cauchy--Schwarz inequality.
Note that $\mathcal{E}_j(f)^{1/2} = \|(I - A_j)^{1/2} f\|$ is a seminorm, and that, since $(I - A_j)^2 \leq I - A_j \leq I - A_j^2 \leq I$,
\[\eqtag \label{eqEj}
\|(I - A_j) f\|^2 \leq \mathcal{E}_j(f) \leq \|f\|^2 - \|A_j f\|^2 \leq \|f\|^2.
\]

For the first inequality, let $h_n := f$ and $h_{j-1} := A_j h_j$ for $1 \leq j \leq n$, so that $h_0 = A_1 \cdots A_n f$.
By telescoping,
\[\eqtag \label{eqtelh}
\|f\|^2 - \|h_0\|^2 = \sum_{j=1}^{n} \big( \|h_j\|^2 - \|A_j h_j\|^2 \big) \overset{\eqref{eqEj}}{\geq} \sum_{j=1}^{n} \mathcal{E}_j(h_j).
\]
Moreover, $f - h_k = \sum_{j > k} (I - A_j) h_j$, again by telescoping, so that, for every $k$,
\begin{align*}
\mathcal{E}_k(f)^{1/2}
&\overset{\textrm{TI}}{\leq} \mathcal{E}_k(h_k)^{1/2} + \mathcal{E}_k(f - h_k)^{1/2}
\overset{\eqref{eqEj}}{\leq} \mathcal{E}_k(h_k)^{1/2} + \|f - h_k\| \\
&\overset{\textrm{TI}}{\leq} \mathcal{E}_k(h_k)^{1/2} + \sum_{j>k} \|(I - A_j)h_j\|
\overset{\eqref{eqEj}}{\leq} \sum_{j \geq k} \mathcal{E}_j(h_j)^{1/2}.
\end{align*}
By CS, $\mathcal{E}_k(f) \leq n \sum_{j} \mathcal{E}_j(h_j)$ for every $k$; summing over $k$ and using \eqref{eqtelh} gives
\[\eqtag \label{eqEfh}
\mathcal{E}(f) \leq n^2 \sum_{j=1}^{n} \mathcal{E}_j(h_j) \leq n^2 \big( \|f\|^2 - \|h_0\|^2 \big).
\]
Finally, as $\|h_0\| \leq \|f\|$,
\begin{align*}
\mathrm{Re}\, \langle (I - A_1 \cdots A_n) f \mid f \rangle
&= \|f\|^2 - \mathrm{Re}\, \langle h_0 \mid f \rangle
\overset{\textrm{CS}}{\geq} \|f\|^2 - \|h_0\|\, \|f\| \\
&\geq \tfrac{1}{2} \big( \|f\|^2 - \|h_0\|^2 \big)
\overset{\eqref{eqEfh}}{\geq} \tfrac{1}{2n^2}\, \mathcal{E}(f).
\end{align*}

For the second inequality, let $u_0 := g$ and $u_j := A_j u_{j-1}$ for $1 \leq j \leq n$.
By telescoping, $I - A_1 \cdots A_n = \sum_{j=1}^{n} A_1 \cdots A_{j-1} (I - A_j)$, and, as the $A_i$ are self-adjoint, $\langle A_1 \cdots A_{j-1} (I - A_j) f \mid g \rangle = \langle (I - A_j)^{1/2} f \mid (I - A_j)^{1/2} u_{j-1} \rangle$.
Hence, letting $a_j := \mathcal{E}_j(u_{j-1})^{1/2}$,
\[\eqtag \label{eqaj}
\big| \langle (I - A_1 \cdots A_n) f \mid g \rangle \big| \overset{\textrm{CS}}{\leq} \sum_{j=1}^{n} \mathcal{E}_j(f)^{1/2} a_j \overset{\textrm{CS}}{\leq} \mathcal{E}(f)^{1/2} \Big( \sum_{j=1}^{n} a_j^2 \Big)^{1/2}.
\]
As $g - u_{j-1} = \sum_{i<j} (I - A_i) u_{i-1}$, by telescoping once more,
\begin{align*}
a_j
&\overset{\textrm{TI}}{\leq} \mathcal{E}_j(g)^{1/2} + \mathcal{E}_j(g - u_{j-1})^{1/2}
\overset{\eqref{eqEj}}{\leq} \mathcal{E}_j(g)^{1/2} + \|g - u_{j-1}\| \\
&\overset{\textrm{TI}}{\leq} \mathcal{E}_j(g)^{1/2} + \sum_{i<j} \|(I - A_i) u_{i-1}\|
\overset{\eqref{eqEj}}{\leq} \mathcal{E}_j(g)^{1/2} + \sum_{i<j} a_i.
\end{align*}
By induction on $j$, this gives
\[
a_j \leq \sum_{i \leq j} 2^{j-i}\, \mathcal{E}_i(g)^{1/2} \leq 2^{j} \sum_{i=1}^{n} \mathcal{E}_i(g)^{1/2} \overset{\textrm{CS}}{\leq} 2^{j} \sqrt{n}\, \mathcal{E}(g)^{1/2}.
\]
Plugging this into \eqref{eqaj} and using $\sum_{j=1}^{n} 4^j \leq 4^{n+1}/3$ concludes the proof.
\end{proof}

\begin{prop}\label{proprecP}
Let $P$ be a symmetric finitely supported probability measure on $\Gm$ whose support generates $\Gm$, let $c \in \Gm$ and $t \in ]0,1[$, and assume that the centraliser $H$ of $c$ is $P$-recurrent.
Let $\mu_t := tP + (1-t)\delta_c$ and $\mu'_t := tP + (1-t)\varsigma_c$, where $\varsigma_c := \tfrac{1}{2}(\delta_c + \delta_{c^{-1}})$.
\begin{itemize}
\item[(i)] $H$ is $\mu'_t$-recurrent; in particular, if $c$ is an involution, $H$ is $\mu_t$-recurrent, as $\mu_t = \mu'_t$.
\item[(ii)] If $c$ has finite order, then $H$ is $(\delta_c * S_t)$-recurrent, where $S_t := (1-t) \sum_{i \geq 0} (tP)^i$ is as in Proposition \ref{propS}.
\end{itemize}
\end{prop}
\begin{proof}
Let $X := H\backslash\Gm$, and recall from Section \ref{s-recext} that, for a non-degenerate probability measure $\nu$ on $\Gm$, $H$ is $\nu$-recurrent if and only if the random walk induced by $\nu$ on $X$, whose steps are $x \mapsto xY^{-1}$ with $Y$ of law $\nu$, is recurrent.
The counting measure on $X$ is invariant for all these random walks, as right multiplications are bijections of $X$, and, if $\nu$ is symmetric, the induced random walk is reversible, with Dirichlet form $\mathcal{E}_\nu$ (see Lemma \ref{lemdir}).
Two reversible random walks on $X$ whose Dirichlet forms are comparable have comparable capacities, and are therefore simultaneously recurrent or transient (Woess \cite[Theorem 2.12 and Corollary 2.14]{Woe}; see Lemma \ref{lemdir}(a)).
Let $K$ be the transition operator of the random walk induced by $P$; it is self-adjoint on $\ell^2(X)$, and this random walk is recurrent by assumption.
Let $L$ be the word length of $c$ with respect to $\spt P$, which is symmetric and generates $\Gm$.

(i) The measure $\mu'_t$ is symmetric, and
\[
\mathcal{E}_{\mu'_t}(\varphi) = t\, \mathcal{E}_P(\varphi) + \tfrac{1-t}{2} \sum_{x \in X} |\varphi(xc^{-1}) - \varphi(x)|^2.
\]
By \eqref{eqpath}, applied to $\gm = c^{-1}$, whose word length is $L$, the last sum is at most $\frac{2L^2}{\min_{\spt P} P}\, \mathcal{E}_P(\varphi)$, so that
\[
t\, \mathcal{E}_P \leq \mathcal{E}_{\mu'_t} \leq C\, \mathcal{E}_P, \qquad \textrm{where } C := t + \frac{(1-t)\, L^2}{\min_{\spt P} P}.
\]
Hence the random walk induced by $\mu'_t$ is recurrent.

(ii) Let $n$ be the order of $c$, let $T$ be the transition operator of the random walk induced on $X$ by $S_t$, and let $\sigma$ be the unitary operator of $\ell^2(X)$ defined by $(\sigma \varphi)(x) := \varphi(xc^{-1})$.
A step of law $S_t$ is a product of a geometric number of independent steps of law $P$, so that $T = (1-t) \sum_{i \geq 0} t^i K^i = (1-t)(I - tK)^{-1}$.
As the spectrum of $K$ lies in $[-1,1]$, $T$ is self-adjoint with spectrum in $[\tfrac{1-t}{1+t}, 1]$; moreover $I - T = t(I - K)(I - tK)^{-1}$, so that $\tfrac{t}{1+t} \mathcal{E}_P \leq \langle (I - T)\varphi \mid \varphi \rangle \leq \tfrac{t}{1-t} \mathcal{E}_P$.
The steps of the random walk induced by $\delta_c * S_t$ are $x \mapsto xB^{-1}c^{-1}$, with $B$ of law $S_t$, so that its transition operator is $T\sigma$.
Since $\sigma^n = I$, its $n$-step transition operator is
\[
(T\sigma)^n = A_0 A_1 \cdots A_{n-1}, \quad \textrm{where } A_j := \sigma^j T \sigma^{-j}.
\]
The operators $A_j$ satisfy the assumptions of Lemma \ref{lemprodpos}.
Moreover, $\langle (I - A_j)\varphi \mid \varphi \rangle = \langle (I - T) \sigma^{-j}\varphi \mid \sigma^{-j}\varphi \rangle$ is comparable to $\mathcal{E}_P(\sigma^{-j}\varphi) = \mathcal{E}_{c^{-j} P c^{j}}(\varphi)$, which is in turn comparable to $\mathcal{E}_P(\varphi)$, by Lemma \ref{lemdir}(b), applied to the pairs of measures $(P, c^{-j} P c^{j})$ and $(c^{-j} P c^{j}, P)$.
Hence, by Lemma \ref{lemprodpos}, the Dirichlet form $\mathcal{D}(\varphi) := \mathrm{Re}\, \langle (I - (T\sigma)^n)\varphi \mid \varphi \rangle$ of the $n$-step random walk is comparable to $\mathcal{E}_P$, and this random walk satisfies the sector condition $|\langle (I - (T\sigma)^n)\varphi \mid \psi \rangle| \leq C' \mathcal{D}(\varphi)^{1/2} \mathcal{D}(\psi)^{1/2}$.
By the Dirichlet principle for non-reversible Markov chains of Gaudilli\`ere \& Landim \cite{GL}, the capacities of a Markov chain satisfying a sector condition are comparable to those of its symmetrisation, whose Dirichlet form is $\mathcal{D}$.
Therefore the capacities of the $n$-step random walk are comparable to those of the random walk induced by $P$, so that it is recurrent.
A fortiori, $H$ is $(\delta_c * S_t)$-recurrent, as $\sum_{m \geq 0} (\delta_c * S_t)^m(H) \geq \sum_{k \geq 0} (\delta_c * S_t)^{nk}(H) = \infty$.
\end{proof}

\begin{cor}\label{correc}
Let $P$ be a probability measure on $\Gm$, let $c \in \spt(P^k)$ for some $k \geq 1$ and $t \in ]0,1[$, and assume that $c$ has a $P$-recurrent centraliser.
Then $\BH_P(\Gm) \subset \BH_{\mu_t}(\Gm) \cap \BH_{\mu'_t}(\Gm)$, where $\mu_t$ and $\mu'_t$ are as in Proposition \ref{proprecP}.
If moreover $P$ is symmetric and finitely supported with support generating $\Gm$, then $\BH_{\mu'_t}(\Gm) = \BH_P(\Gm)$, and, if $c$ has finite order, $\BH_{\mu_t}(\Gm) = \BH_P(\Gm)$.
In particular, $\Gm$ is then $\mu'_t$-Liouville (resp. $\mu_t$-Liouville) if and only if it is $P$-Liouville.
\end{cor}
\begin{proof}
By Lemma \ref{lemmart}, every bounded $P$-harmonic function $f$ satisfies $f*\delta_c = f$, hence also $f*\delta_{c^{-1}} = f$, so that $f*\mu_t = f*\mu'_t = f$ (this was already observed before Proposition \ref{proprec}).
Under the additional assumptions, Proposition \ref{proprecP} and Lemma \ref{lemmart}, applied to $\mu'_t$, resp. to $\delta_c * S_t$ (recall that $\BH_{\mu_t}(\Gm) = \BH_{\delta_c * S_t}(\Gm)$ by Proposition \ref{propS}, and that $(\delta_c * S_t)(c) = S_t(e) > 0$), show that every bounded $\mu'_t$-harmonic, resp. $\mu_t$-harmonic, function $f$ satisfies $f*\delta_c = f*\delta_{c^{-1}} = f$.
Then $f = f*\mu'_t = t f*P + (1-t) f$, resp. $f = f*\mu_t = t f*P + (1-t) f$, so that $f$ is $P$-harmonic.
\end{proof}
In other words, adding $c$ then does not change the bounded harmonic functions at all; this is Theorem \ref{teomainC} of the introduction.
When the centraliser of $c$ is not recurrent for the relevant random walks, neither Proposition \ref{proprec} nor Corollary \ref{correc} gives any information.

To finish on an amusing note, if $P$ is ``lazy enough'', \ie $P = a\delta_e + (1-a) Q$ for some measure $Q$ and some $a \in ]\tfrac{1}{2},1[$,
then it is possible to go back in time: $P^{-1}$ is the sum $\tfrac{1}{a} \sum_{i \geq 0} \Big( \tfrac{a-1}{a} Q \Big)^i$.
However, since $\tfrac{a-1}{a}$ is negative, this is not a (positive) measure.
The sum can be split in an obviously positive part (whose $\ell^1$-norm is $\tfrac{a}{2a-1}$) and a negative part (whose $\ell^1$-norm is $\frac{1-a}{2a-1}$).
These two parts never cancel out to leave a norm of 1.
Indeed, the norm of $f \mapsto f*P^{-1}$ on $\ell^\infty(\Gm)$ is $\|P^{-1}\|_{\ell^1}$: it is at most $\|P^{-1}\|_{\ell^1}$ by H\"older's inequality, and this value is attained at $e$ by the function $f(y) := \mathrm{sign}\, P^{-1}(y^{-1})$.
It is larger than one unless $Q = \delta_e$.
Indeed, the total mass of $P^{-1}$ is 1, since the total mass is multiplicative under convolution and $P*P^{-1} = \delta_e$; so if $\|P^{-1}\|_{\ell^1} = 1$, then $P^{-1}$ is a probability measure.
But if $\alpha$ and $\beta$ are probability measures with $\alpha * \beta = \delta_e$, then $\spt \alpha \cdot \spt \beta = \{e\}$, so that $\alpha$ and $\beta$ are Dirac masses; this would force $P = \delta_e$, \ie $Q = \delta_e$.

\end{document}